\documentclass[12pt]{amsart}
\usepackage{enumerate}
\usepackage{amsmath}
\usepackage{amsthm}
\numberwithin{equation}{section}
\input epsf
\usepackage{epsfig}
\begin{document}
\def\A{\mathcal{A}}
\def\id{\mathrm{id}}
\def\Ima{\operatorname{Im}}
\def\Rea{\operatorname{Re}}
\def\Aut{\mathrm{Aut}}
\def\length{\operatorname{length}}
\def\o{\mathrm{o}}
\def\cross{\times}
\def\BL{\operatorname{BL}}
\def\Z{\mathbf{Z}}
\def\N{\mathbf{N}}
\def\R{\mathbf{R}}
\def\C{\mathbf{C}}
\def\U{\mathbf{U}}
\def\H{\mathbf{H}}
\def\bC{\mathbf{\overline{C}}}
\def\mod{\operatorname{mod}}
\def\sgn{\operatorname{sgn}}
\def\tr{\operatorname{tr}}
\def\const{\operatorname{const}}
\def\D{{\Omega}}
\newtheorem{theorem}{Theorem}[section]
\newtheorem{lemma}{Lemma}[section]
\newtheorem{corollary}{Corollary}[section]
\newtheorem{problem}{Problem}
\newtheorem*{theorem1.1}{Theorem $\mathbf{1.1^\prime}$}
\newtheorem*{theoremA}{Theorem A}
\newtheorem*{conj}{Conjecture}
\theoremstyle{remark}
\newtheorem*{remark}{Remark}
\newtheorem*{ack}{Acknowledgment}
\title{Goldberg's constants}
\dedicatory{Dedicated to the memory of A. A. Goldberg}
\author{Walter Bergweiler}
\address{Mathematisches Seminar, Christian-Albrecht-Universit\"at zu Kiel
Ludewig-Meyn-Stra{\ss}e 4, 24098 Kiel, Germany}
\email{bergweiler@math.uni-kiel.de}
\author{Alexandre Eremenko}
\thanks{The first author was supported by the Deutsche
Forschungsgemeinschaft, Be 1508/7-2,
and the ESF Networking Programme HCAA, and
the second author was supported by NSF grant DMS-1067886}
\address{Department of Mathematics, Purdue University,
West Lafayette, IN 47907, USA}
\email{eremenko@math.purdue.edu}
\begin{abstract} 
We study two extremal problems of geometric function theory
introduced by A.\ A.\ Goldberg in 1973.
For one problem we find the exact solution,
and for the second one we obtain partial results.
In the process we study the lengths of
hyperbolic geodesics in the twice punctured
plane, prove several results about them and make a conjecture.
Goldberg's problems have important applications to control theory.
\end{abstract}
\renewcommand{\thefootnote}{}
\maketitle
\section{Introduction}\label{introduction}
Goldberg~\cite{Goldberg} studied a class of extremal problems
for meromorphic functions. 
Let $F_0$ be the set of all holomorphic functions $f$ defined in the rings
$$\{ z:\rho(f)<|z|<1\},$$
omitting $0$ and~$1$, and such 
that the indices of the curve $f(\{ z:|z|=\sqrt{\rho(f)}\})$
with respect to $0$ and $1$ are non-zero and distinct.

Let $F_1\subset F_0$ be the subclass consisting of 
functions meromorphic
in the unit disk~$\U$. Functions in $F_1$ can be described
as meromorphic functions in $\U$ with
the property that the numbers of
preimages of~$0$, $1$ and $\infty$, counted with multiplicities, 
are all finite and pairwise distinct.

Let $F_2,F_3,F_4$ be the subclasses of $F_1$ consisting of
functions holomorphic in the unit disk, rational functions and polynomials,
respectively.
For $f$ in any of these classes $F_j$, $1\leq j\leq 4$,
we define $\rho(f)$ as
$$\rho(f)=\sup\{|z|:f(z)\in\{0,1,\infty\}\}.$$
Goldberg's constants are
$$A_j=\inf_{F_j}\rho(f),\quad 0\leq j\leq 4.$$
Goldberg credits
 the problem of minimizing $\rho(f)$ to E. A. Gorin.
He proved that
\begin{equation}\label{gold}
0<A_0=A_1=A_3<A_2=A_4,
\end{equation}
and showed that there exist extremal functions for $A_0$ and $A_2$,
but extremal functions for $A_1, A_3$ or $A_4$ do not exist. 
He also proved the estimates
$$A_0<0.0091\quad\mbox{and}\quad 0.0000038<A_2<0.0319.$$
In view of (\ref{gold}), we consider only $A_0$ and $A_2$.

The constants $A_0$ and $A_2$ are important for several reasons.
\begin{problem}
Which triples of
non-negative
divisors in $\U$ of finite degree are divisors of zeros,
poles and $1$-points
of a meromorphic function in $\U$ ?
\end{problem}
The constants $A_0$ and
$A_2$ give the
only general restrictions for this problem that are known to us.
\begin{problem}
Let $\phi_1,\phi_2,\ldots,\phi_n$ be rational functions restricted on~$\U$.
Does there exist a meromorphic function $f$ in $\U$ which avoids
$\phi_1,\ldots,\phi_n$ ?
\end{problem}

Avoidance means that the graphs of $f$ and
$\phi_j$ are disjoint subsets of $\U\times\bC$,
that is $f(z)\neq\phi_j(z)$ for $z\in \U$. If the graphs of the $\phi_j$
are pairwise disjoint, then such a function
$f$ exists; this is a famous result of
Slodkowski \cite[Lemma~2.1]{Slod};
see also~\cite{Chirk}. If $n=3$ and the graphs of two functions
$\phi_1$ and $\phi_2$ are disjoint, then the avoidance problem is equivalent
to Problem 1 for holomorphic functions~\cite{Blondel}. 

The avoidance problem is important for control theory: it is equivalent
to the problem of
simultaneous stabilization of several single input -- single output
linear systems, see~\cite{Blondel,1,Burke,E} and references therein.

\medskip

In this paper we find the exact value of $A_0$
and some related constants which are then used in our investigation of
$A_2$, on which we only have partial results.

The first explicit lower bound for $A_0$ was found by
Jenkins~\cite{Jenkins} who stated his result as
\begin{equation}\label{je}
A_0\geq 0.00037008.
\end{equation}
Blondel, Rupp and Shapiro~\cite{1} proved that $A_2>10^{-5}$, then
Batra~\cite{B1,Batra} improved this to 
$A_2>0.0012$.

In section \ref{jenkins} we give the precise value:

\begin{theorem}\label{A_0}
$$
A_0=J:=
\exp\left(-\frac{\pi^2}{\log(3+2\sqrt{2})}\right)\approx 0.003701599.
$$
\end{theorem}

We will see that Theorem~\ref{A_0} is equivalent to the following result,
which is
essentially well known.
Let $\D=\C\backslash\{0,1\}$. A closed curve $\gamma$ in $\D$ is called
{\em peripheral} if it can be continuously deformed in $\D$ to a point in $\bC$
(possibly to a puncture $0,1$ or $\infty$). We recall that the hyperbolic
metric is a complete Riemannian metric of constant curvature $-1$.
\begin{theorem1.1}
The smallest hyperbolic length of a non-peripheral
curve in $\D$ is $2\log(3+2\sqrt{2})$.
\end{theorem1.1}

Theorem ${1.1^\prime}$ follows from~\cite[Theorem C]{Schmutz} or~\cite{Baribaud}.

\begin{figure}[htb]
\begin{center}
\includegraphics[height=3.5cm]{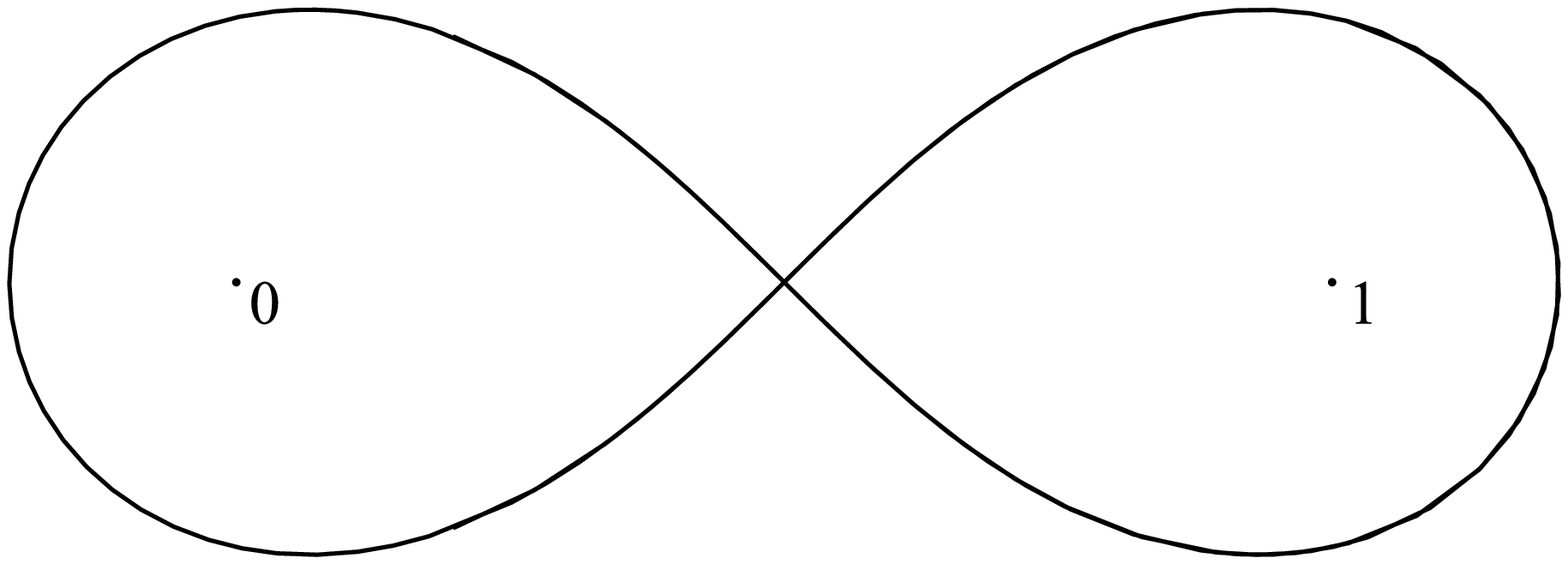}
\quad 
\includegraphics[height=3.5cm]{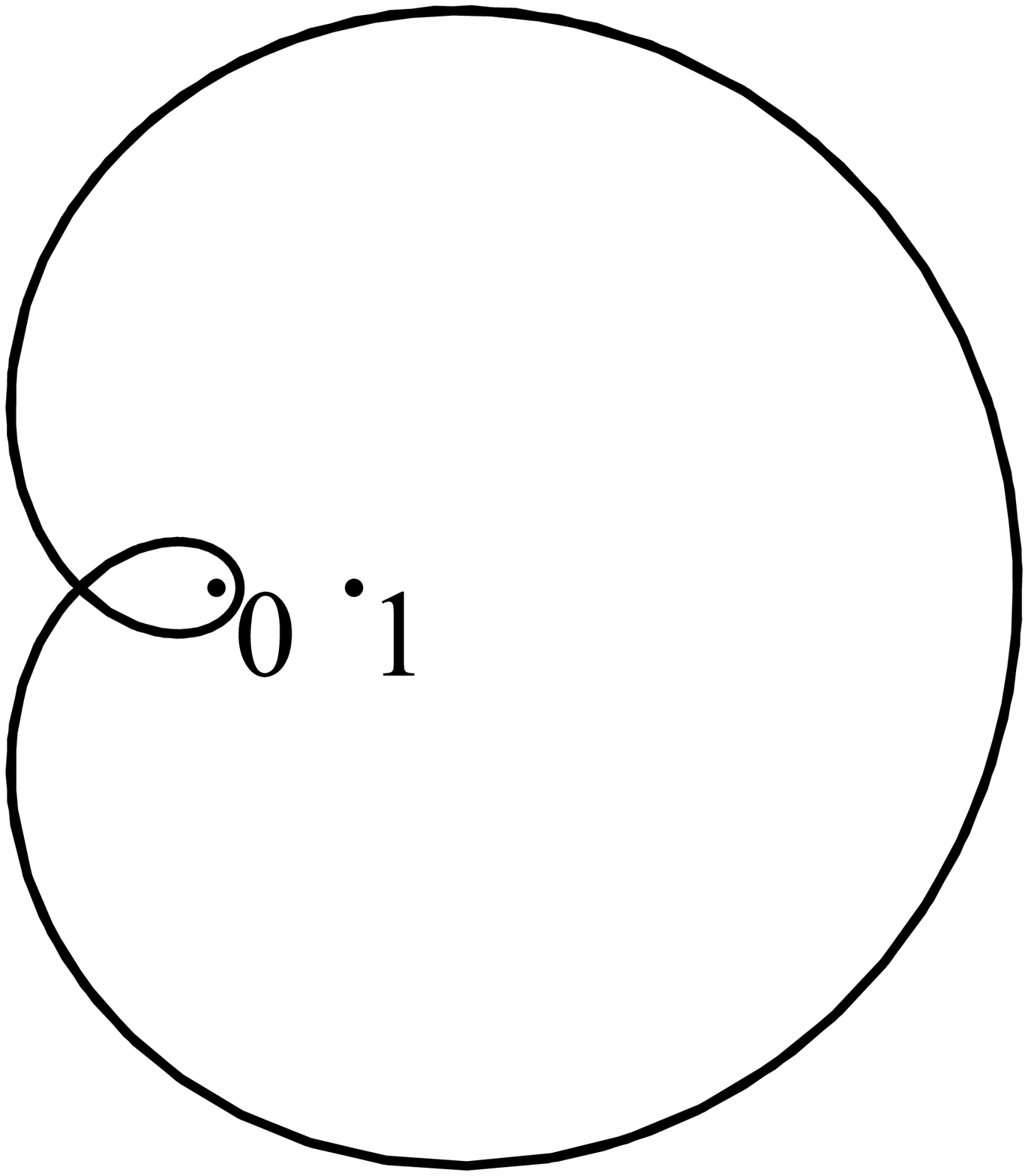}
\caption{Two shortest non-peripheral curves in
the twice punctured plane.}
\label{fig1}
\end{center}
\end{figure}

The inequality $A_0\geq J$ was actually stated
in Jenkins's
paper~\cite{Jenkins}, and this lower bound with the correct value of $J$
contradicts his own
upper bound $0.00149$, but he calculated the numerical value of $J$
incorrectly to obtain (\ref{je}).   
Moreover, he did not notice that his method gives $A_0=J$.
The details of the computation
of the (incorrect) upper bound are omitted in Jenkins's paper.
Because of these and other mistakes in~\cite{Jenkins},
we give in section \ref{jenkins} a complete proof of Theorem~\ref{A_0}.
Our argument in section \ref{jenkins} is essentially
the same as that of Jenkins;
we only correct his mistakes.

Let $\sigma$ be the circle $\{ z:|z|=\sqrt{\rho(f)}\}$, with 
counterclockwise orientation, and let $\gamma_f$ be the 
image of $\sigma$ under $f$.
The definition of $F_0$ implies that 
$\gamma_f$
is non-peripheral for $f\in F_0$.
Only this property is used in Goldberg's theorems and in
Theorem~\ref{A_0}. However,
in applications to Problems 1 and 2 above, 
the numbers of $0$-, $1$- and $\infty$-points
of $f$ in the unit disk are prescribed, and nothing is known a priori
about the nature of the curve $\gamma_f$.
This suggests the following definitions.

For distinct, non-zero integers $N_0$ and $N_1$
we consider the subclass $F_0(N_0,N_1)$ of $F_0$
consisting of those functions
for which the indices of the curve $\gamma_f$ about $0$ and~$1$
are $N_0$ and $N_1$, respectively.
Then we define
$$A_0(N_0,N_1)=\inf\{\rho(f):f\in F_0(N_0,N_1)\}.$$
The classes $F_j(N_0,N_1)$ and the
constants $A_j(N_0,N_1)$, $1\leq j\leq 4$, are defined similarly.
Evidently, $A_j(N_0,N_1)=A_j(N_1,N_0)$.
One can show that
$$A_0(N_0,N_1)=A_1(N_0,N_1)=A_3(N_0,N_1)<A_2(N_0,N_1)=A_4(N_0,N_1)$$
in the same way as Goldberg proved (\ref{gold}).
Thus it again suffices to consider $A_0(N_0,N_1)$ and $A_2(N_0,N_1)$.

Note that if $f\in F_0(N_0,N_1)$, then
$$\frac{1}{f}\in F_0(-N_0,N_1-N_0)
\quad\text{and}\quad
f(\rho(f)/z)\in F_0(-N_0,-N_1).$$
 Thus
$$A_0(N_0,N_1)=A_0(-N_0,N_1-N_0)
\quad\text{and}\quad
A_0(N_0,N_1)=A_0(-N_0,-N_1).$$
Together with $A_0(N_0,N_1)=A_0(N_1,N_0)$
this implies that 
we may restrict to the case $N_0>N_1>0$ in our 
investigation of $A_j(N_0,N_1)$ not only if
$j=2$, but also if $j=0$.
Moreover, we have 
\begin{equation}\label{N0-N1}
A_0(N_0,N_1)=A_0(N_0,N_0-N_1).
\end{equation}

In section \ref{fine} we will prove the following result.

\begin{theorem}\label{A_0(N_0,N_1)}
For $N_0,N_1\in\N$ let $N=N_0+N_1$ and put
$N^*=N$ when $N$ is odd and $N^*=2N-3$ when $N$ is even. Then
$$A_0(N_0,N_1)\geq \exp\left(-\frac{\pi^2}{\cosh^{-1}(N^*)}\right).$$
This estimate is best possible (i.e., there exist extremal functions)
for all $N\geq 3$.
\end{theorem}

A corollary of Theorem~\ref{A_0(N_0,N_1)} is
that
\begin{equation}\label{10}
A_2(N_0,N_1)\geq\exp\left(-\frac{\pi^2}{\log(2\max\{ N_0,N_1\})}\right).
\end{equation}
This improves the result of~\cite{1} which in our notation
says that
$$
A_2(N_0,N_1)\geq \exp\left\{-\left(1+\frac{2}{\pi e}\right)\frac{\pi^2}{\log
\min\{
N_0,N_1\}}\right\}.
$$

Our method allows in principle to compute the exact value of $A_0(N_0,N_1)$
for any given $N_0,N_1$. 
The algorithm is described in section \ref{fine}.
For $N=3,4,5,$ we obtain
$N^*=3,5,5$. For $(N_0,N_1)=(2,1),(3,1),(3,2)$ we have equality in
our estimate for $A_0(N_0,N_1)$. We obtain $A_0(2,1)=A_0$ and
$$
A_0(3,1)=
A_0(3,2)=\exp\left(-\frac{\pi^2}{\log(5+2\sqrt{6})}\right)\approx
0.013968.
$$
However, as apparent already from \eqref{N0-N1},
the constant $A_0(N_0,N_1)$ is not a function of the sum $N_0+N_1$ only,
and we have
\begin{equation}\label{14}
A_0(4,1)=A_0(4,3)=\exp\left(-\frac{\pi^2}{\log(7+4\sqrt{3})}\right)\approx
0.023585.
\end{equation}
Finding the constants $A_0(N_0,N_1)$ has the following geometric interpretation.
Consider the set of all closed curves
in $\D$ with indices $N_0,N_1$ with respect 
to $0$ and~$1$. Find the minimal hyperbolic length of a curve
in this class. Some examples of minimal curves can be seen in 
Figures~\ref{fig1} and~\ref{fig2}.
\begin{figure}[htb]
\begin{center}
\includegraphics[height=4.5cm]{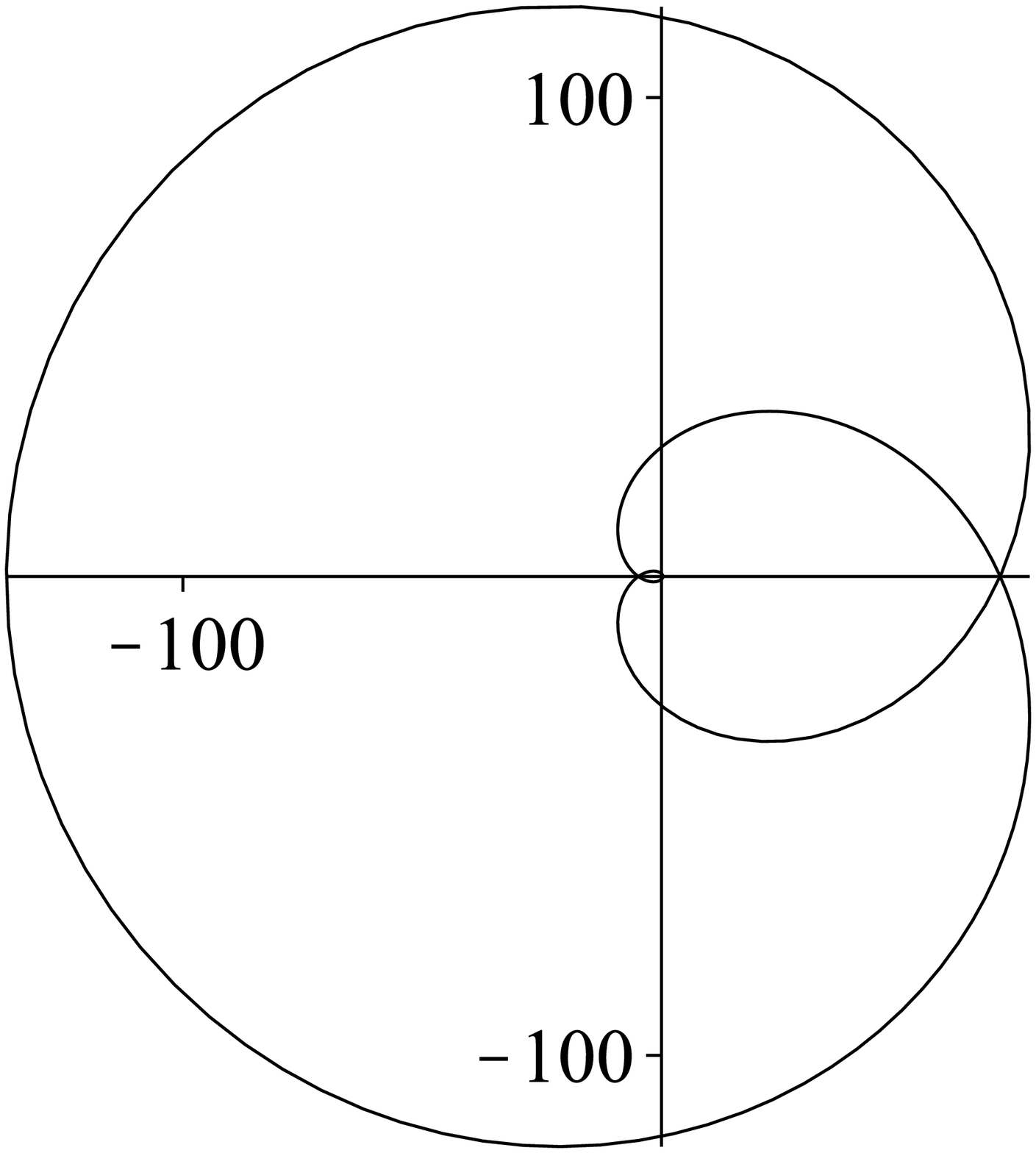}
\quad
\includegraphics[height=4.5cm]{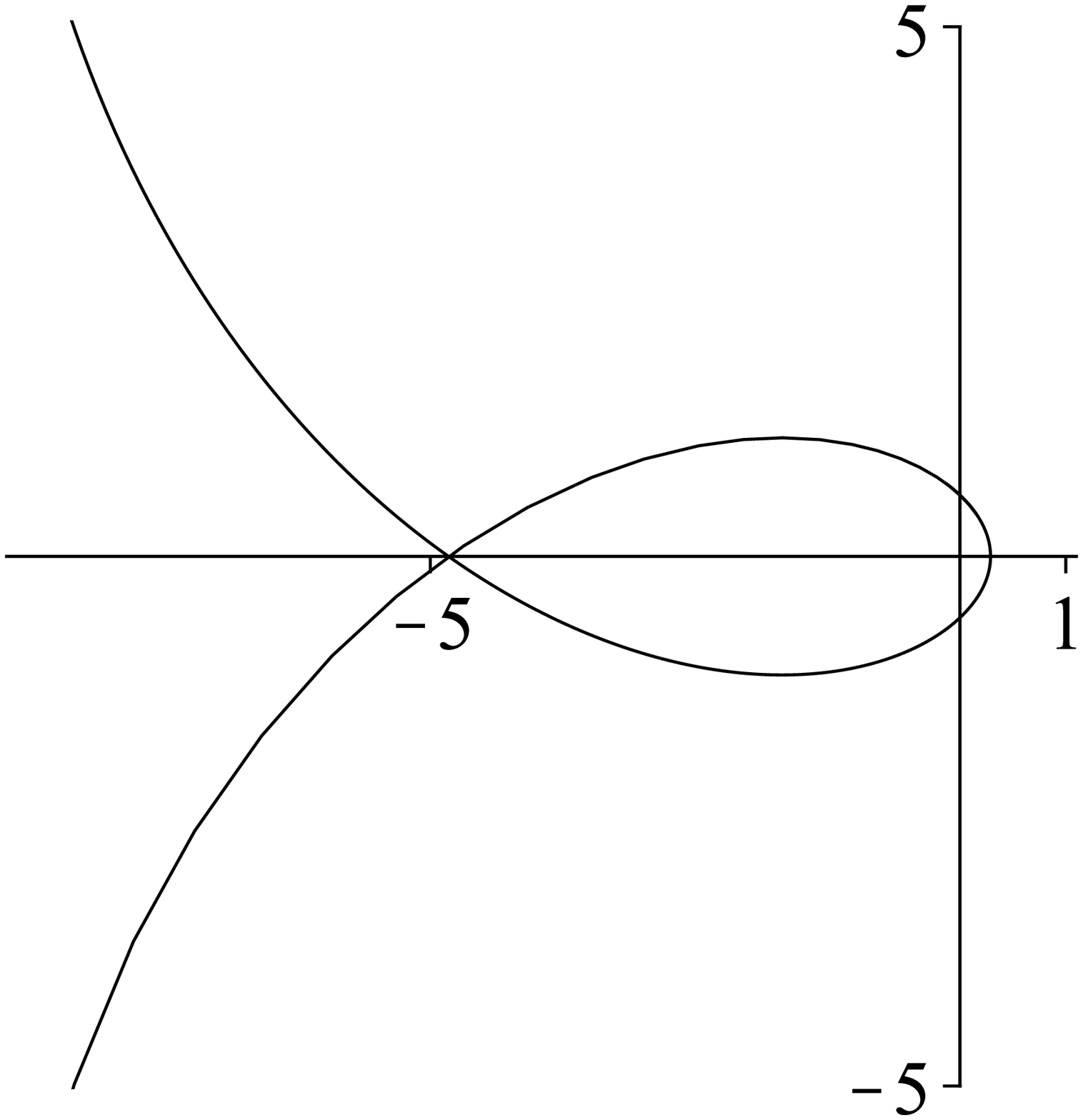}
\quad
\includegraphics[height=4.5cm]{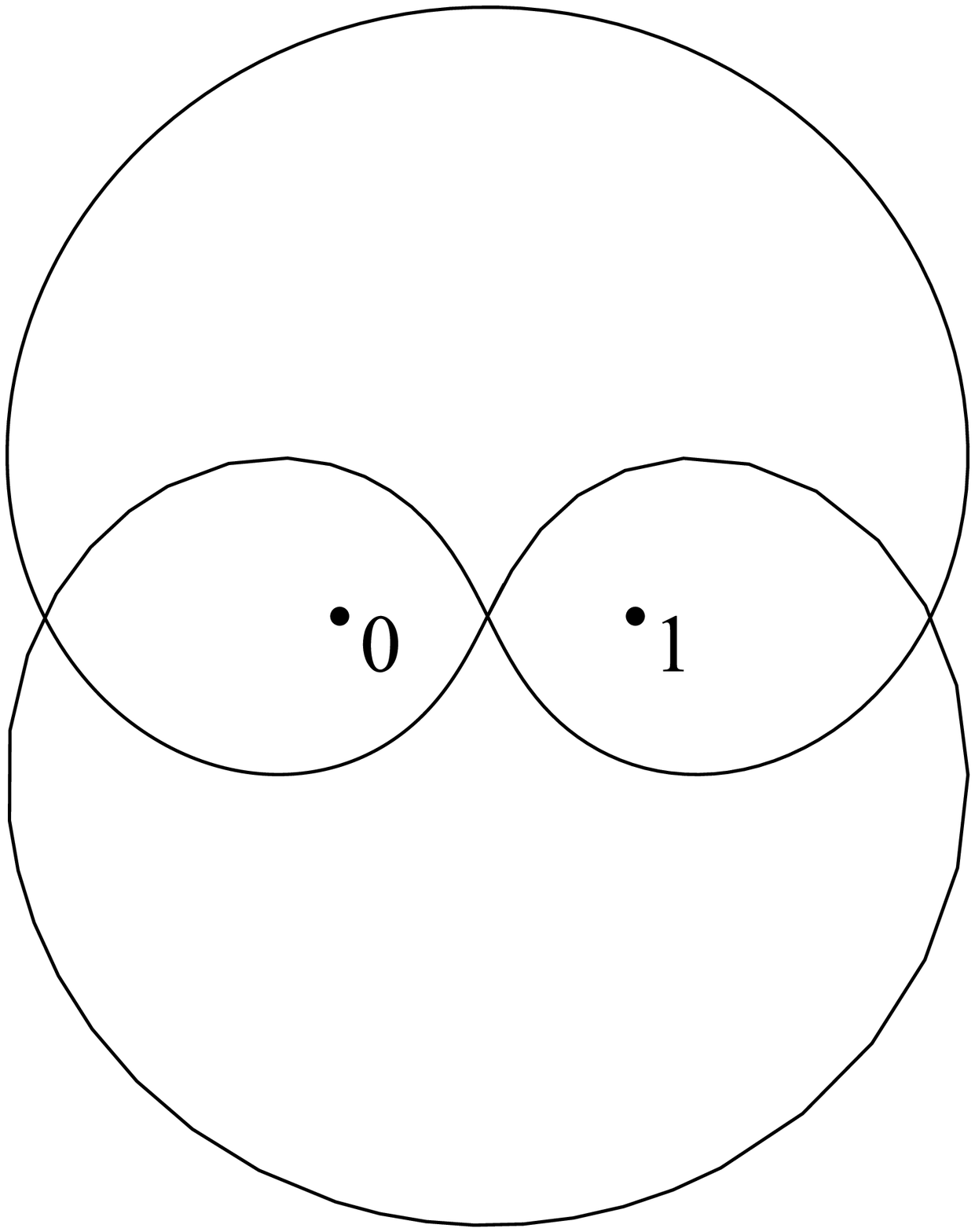}
\caption{The shortest curve 
of index $3$ around $0$ and index $2$ around $1$,
with magnification of detail, and the shortest
curve homotopic to the commutator of the loops around $1$ and~$0$.}
\label{fig2}
\end{center}
\end{figure}

In section \ref{conjectures} we state a formula for and 
a conjecture about traces of elements
of the principal con\-gru\-e\-nce subgroup $\Gamma(2)$ of the modular group.
We found this conjecture while experimenting
with traces trying to prove Theorem~\ref{A_0(N_0,N_1)}, but in our opinion
this conjecture is of independent interest. 

Goldberg's theorem says that $A_2>A_0$, but no explicit estimate of $A_2$
from below better than $A_0$ was available. Using Theorem~\ref{A_0(N_0,N_1)}
and a result of Dubinin~\cite{Dubinin}
we can obtain such a bound.

\begin{theorem}\label{AAA} For every function
$f\in F_2(N_0,N_1)$,
we have
\begin{equation}\label{AAA1}
\rho(f)\geq \left(1+\sqrt{1-16A_0(N_0,N_1)^{2q}}\right)^{2/q}A_0(N_0,N_1),
\end{equation}
where $q$ is the cardinality of $f^{-1}(\{0,1\})$. Moreover,
\begin{equation}\label{AAA2}
A_2\geq \left(1+\sqrt{1-16A_0^6}\right)^{2/3}A_0\approx 0.00587465.
\end{equation}
\end{theorem}

Now we describe the conjectured extremal function for $A_2$.
A function $f\in F_2$ can be considered as
a holomorphic map 
\begin{equation}\label{cov}
f:\U\backslash f^{-1}(\{0,1\})\to\C\backslash\{0,1\}.
\end{equation}
Choose a point $z_0\in \U\backslash f^{-1}(\{0,1\})$ and let $w_0=f(z_0)$.
Then $f$ defines a homomorphism $f_*$ of the fundamental groups:
$$f_*:\pi(z_0, \U\backslash f^{-1}(\{0,1\}))\to\pi(w_0,\C\backslash\{0,1\}).$$
The fundamental group $\pi(w_0,\D),$ where
$\D=\C\backslash\{0,1\},$ is a free group
on two generators $A$ and $B$ which are simple counterclockwise
loops around $0$ and~$1$.

The image of the homomorphism $f_*$ is a subgroup of 
$\pi(w_0,\D)$ which we denote by $\Gamma(f)$. If we change $z_0$ and $w_0$
we obtain a conjugate subgroup.
If (\ref{cov}) is a covering map, then $f_*$ is injective, so 
$\Gamma(f)$ is isomorphic to $\pi(z_0,\U\backslash f^{-1}(\{0,1\}))$; 
cf.~\cite[Section 9.4]{Ahlfors2}.
We call those functions $f$ for which (\ref{cov}) is a covering map
{\em locally extremal}.

In sections \ref{h} and \ref{a_2} we define a holomorphic function $h$
in the unit disk, real on $(-1,1)$,
with the following properties:

\begin{enumerate}[(a)]
\setlength{\itemindent}{10pt}
\item $h$ has one double zero at the point $-\mu<0$, and no other zeros,

\item $h$ has one simple $1$-point at the point
$\mu$, and no other $1$-points, 

\item $h'(z)\neq 0$ for $|z|<1,\; z\neq-\mu$, 

\item $0,1,\infty$ are the only asymptotic values of~$h$, and

\item $\Gamma(h)$ is generated by $A^2$ and~$B$. 
\end{enumerate}
In other words,
$$h:\U\backslash\{-\mu,\mu\}\to\C\backslash\{0,1\}$$
is a covering map corresponding to the subgroup $\langle A^2,B\rangle$
generated by $A^2$ and~$B$. This map extends to a function in $\U$ 
with a double zero at $-\mu$ and a simple $1$-point at $\mu$.
We will show that $h$ exists and is uniquely defined
by the properties (a)--(e). In particular,
$\mu$ is an absolute constant. Actually for a real function
in the unit disk, properties (a)--(d) imply (e), but we do not
need this fact.

A function $h_0$ with a simple root at~$0$, no other zeros and no $1$-points
in the unit disk, and such that $h_0:\U\backslash\{0\}\to \D$
is a covering map, was studied by Hurwitz
\cite{Hurwitz} and Nehari~\cite{Nehari}. These authors
found several extremal properties
of this function. Our function $h$ and other locally extremal
functions introduced in
section~\ref{h} can be considered as generalizations of this function
of Hurwitz.

Evidently, $A_2\leq\mu$, so we obtain an upper estimate for $A_2$.
In section~\ref{A} we describe an algorithm to compute $\mu$
with any given precision, and obtain the numerical value:

\begin{theorem}\label{A_2}
\quad $A_2\leq\mu\approx 0.0252896$.
\end{theorem}

In section~\ref{last} we study an analytic
representation of our function $h$
and describe an algorithm which permits to compute it. We 
represent $h$ as a composition of the modular function, an
elliptic integral and a special solution of the Lam\'e differential
equation.

We conjecture that $A_2=\mu$.
As supporting evidence we prove 
the following extremal property of~$h$. Let $F_5(m,n)$
be the subclass of $F_2(m,n)$ consisting of functions having one zero
of multiplicity $m$ and one $1$-point of multiplicity $n\neq m$ and put
$$A_5(m,n)=\inf\{ \rho(f):f\in F_5(m,n)\}
\quad\text{and}\quad
A_5=\inf_{m\neq n} A_5(m,n).$$
Evidently $A_5(m,n)=A_5(n,m)$, so it is enough to consider
the case $m>n$. 

\begin{theorem}\label{wk}
Let $1\leq n<m$ and $f\in F_5(m,n)$.
Then 
$\rho(f)\geq\rho(h)=\mu,$
with equality only
for $f(z)=h(e^{i\theta}z)$ with $\theta\in\R$.
In particular,
$A_5=A_5(2,1)=\mu$.
\end{theorem}
We will prove Theorem~\ref{wk} in section~\ref{cs}.

In section~\ref{h} we will actually prove a stronger result.
We show that every function
$f\in F_5(m,n)$ is {\em subordinate} to some locally extremal
function~$g$.
Subordination means that $f=g\circ\omega$,
where $\omega$ is a holomorphic map of $\U$ into itself. So
$\rho(f)\geq\rho(g)$ is a consequence of the Schwarz Lemma.

This approach yields 
functions $h_{m,n}\in F_5(m,n)$ which are extremal for $A_5(m,n)$.
These extremal functions $h_{m,n}$ are defined as  covering maps
$$h_{m,n}:\U\backslash\{-\mu_{m,n},\mu_{m,n}\}\to\C\backslash\{0,1\},$$
where $\mu_{m,n}>0$ and $\Gamma(h_{m,n})$ is the group
generated by $A^m$ and $B^n$.
The function $h_{m,n}$ is holomorphic in $\U$, has a zero of
multiplicity $m$ at $-\mu_{m,n}$ and a $1$-point of multiplicity $n$
at $\mu_{m,n}$.  
We obtain
$A_5(m,n)=\mu_{m,n}$ and, up to rotations,
$h_{m,n}$ is the unique extremal function $A_5(m,n)$.

The computation of the constants $\mu_{m,n}$ is performed with the same
method as our computation of $\mu=\mu_{2,1}=A_5(2,1)$.
Here are some numerical values:
$$\begin{array}{l}A_5(2,1)=0.0252896,\\
A_5(3,1)=0.0849241,\\
A_5(4,1)=0.140571,\\
A_5(3,2)=0.227417,\\
A_5(4,3)=0.290697.\end{array}$$

Theorem \ref{wk} permits to obtain a complete solution of Problem 1
mentioned above in the simplest case of two points.
\begin{theorem}\label{three}
Let $a,b\in \U$.
There exists a holomorphic function $f\in F_2$ 
with $f^{-1}(\{0,1\})=\{ a,b\}$ 
if and only if
\begin{equation}\label{ineq}
\frac{|b-a|}{|1-\overline{a}b|}\geq \frac{2\mu}{1+\mu^2}
\approx 0.050546,
\end{equation}
where $\mu$ is the constant of Theorem~\ref{A_2}
computed in section~\ref{A}.
There exists a rational function $f$ 
with $f^{-1}(\{0,1\})\cap \U=\{ a,b\}$ if and only if
the inequality $(\ref{ineq})$ is strict.
In this case there even exists a polynomial 
with this property.
\end{theorem}
The simplest situation which is not covered by Theorem~\ref{wk}
is the case when $f\in F_2(2,1)$
has one simple zero and two simple $1$-points.
Such a function does not have to be subordinate to any locally
extremal function, 
and we could not prove that $\rho(f)\geq\mu$ in this case.

A problem of control theory dealing with functions in $F_2(2,1)$
having one simple zero and two simple $1$-points is the 
so-called Belgian Chocolate Problem.
We will give some applications of our results and methods 
to this problem in section~\ref{bcp}.

\section{Preliminaries and the exact value of $A_0$}\label{jenkins}

For the background  of this section we refer to~\cite{Ahlfors,Ahlfors2},
but note that what we call \emph{covering} is called
\emph{complete covering} there.
A {\em ring} is a Riemann surface whose fundamental group is isomorphic to
$\Z$. Every ring is conformally equivalent to a region of the form
$$\A=\{ z:0\leq\rho<|z|<\rho'\leq\infty\}.$$
The number 
$$\mod(\A)=\frac{1}{2\pi}\log\frac{\rho'}{\rho}$$
is called the modulus
of the ring.
If $\mod(\A)<\infty$, then the the ring is called {\em non-degenerate}.
For a non-degenerate ring we can always take $\rho'=1$, thus a non-degenerate
ring is equivalent to
\begin{equation}\label{0}
\{ z:\rho(\A)<|z|<1\},\quad\mbox{where}\quad\rho(\A)=
\exp\left(-2\pi \mod(\A)\right).
\end{equation}
Consider the universal covering from the upper half-plane  $\H$
to a non-de\-ge\-ne\-rate ring~$\A$.
The group of this covering is a cyclic subgroup of $\Aut(\H)$
generated by a hyperbolic transformation,
which can be taken to
be
$z\mapsto\lambda z$ for some $\lambda>1$.
It is easy to see that
\begin{equation}\label{1}
\rho(\A)=\exp\left(-\frac{2\pi^2}{\log\lambda}\right).
\end{equation}
The hyperbolic metric is defined in the upper half-plane
by its length element
$$\frac{|dz|}{\Ima z}.$$
It descends from $\H$ to $\A$, and there is a shortest hyperbolic geodesic
in the class of the generator of the fundamental group.
The hyperbolic length of this shortest geodesic is
$$\ell(\A)=\int_i^{i\lambda}\frac{|dz|}{\Ima z}=\log\lambda.$$
Thus for every non-degenerate ring, there exists a shortest
closed geodesic.
It is easy to see that no shortest geodesic exists for a degenerate
ring $\{ 0<|z|<1\}$, while in the other degenerate ring,
$\{ z:0<|z|<\infty\}$ there is no hyperbolic metric, so the notion of
shortest geodesic is not defined.

\medskip

Consider now the region $\D=\C\backslash\{0,1\}$ and fix a point $z_0\in \D$.
The fundamental group $\pi(z_0,\D)$ is a free group on two generators.
Let $\Lambda:\H\to \D$ be the universal covering. The covering group $\Gamma(2)$
is a group of fractional linear transformations isomorphic to the fundamental
group $\pi(z_0,\D)$. So to each element of $\pi(z_0,\D)$ corresponds
a fractional-linear transformation.

The covering $\Lambda$ and the group $\Gamma(2)$
are explicitly constructed as follows:
begin with the region
$$G_0=\{ z:|\Rea z|<1,\;|z-1/2|>1/2,\;|z+1/2|>1/2\}.$$
Let $\Lambda$ be the conformal map of the right half of $G_0$
onto $\H$ with the boundary correspondence
$$(0,1,\infty)\mapsto(1,\infty,0).$$
We note that usually a different boundary correspondence is used,
but the one chosen above turns out to be convenient for our purposes.
The map $\Lambda$ extends to $\H$ by reflections and gives the universal covering
$\Lambda:\H\to\Omega$.
The fractional linear transformations 
$$A(z)=z+2\quad\text{and} \quad B(z)=\frac{z}{-2z+1}$$
perform the pairing of the sides of the quadrilateral $G_0$. They
are free generators of the covering group $\Gamma(2)$.

The generator $A$ corresponds to a simple counterclockwise
loop around $0$ in 
$\Omega$
and $B$ to a simple counterclockwise loop around~$1$.

Fractional-linear transformations mapping $\H$ onto itself are represented by
$2\times 2$ matrices with real entries and determinant~$1$.

With this representation,
$\Gamma(2)$ can be identified
with the so-called {\em principal congruence subgroup of level $2$},
it is the factor group of the group
of all $2\times 2$ matrices $M$ with integer elements
and determinant $1$ over the subgroup
$\{\pm I\}$. It is freely generated by the two matrices which we denote
by the same letters as the two loops described above:
\begin{equation}\label{AB}
A=\left(\begin{array}{cc}1&2\\0&1\end{array}\right)\quad\mbox{and}\quad
B=\left(\begin{array}{rr}1&0\\-2&1\end{array}\right).
\end{equation}
Thus to each element $\gamma$ of $\pi(z_0,\D)$
corresponds a fractional-linear transformation represented by
a pair of matrices $\pm M$. The absolute value of the trace
$|\tr M|$ depends only on the conjugacy class
of $\gamma$ in $\pi(z_0,\D)$. The conjugacy classes in $\pi(z_0,\D)$
are called the {\em free homotopy classes}.

Parabolic elements of $\Gamma(2)$ correspond to closed curves
in $\D$ which can be deformed to a point,
possibly to a puncture. We call these elements {\em peripheral}.
Their matrices are characterized by the property that $|\tr M|=2$.

So to every non-peripheral closed curve $\gamma$ in $\D$ we can associate
a hyperbolic element $\phi\in\Gamma(2)$, a ring $\A=\H/\langle\phi\rangle$
and a pair of matrices $\pm M$. 
Then $\ell(\A)$ is the hyperbolic length
of the shortest curve in the free homotopy class of $\gamma$, and
we have the formulas
\begin{equation}\label{4}
\rho(\A)=
\exp\left(-\frac{\pi^2}{\cosh^{-1}\left(|\tr(M)|/2 \right)}\right)
\end{equation}
and
\begin{equation}\label{minlength}
\ell(\A)
=2\cosh^{-1}\left(\frac{|\tr(M)|}{2}\right)=
2\log\left(\frac{|\tr(M)|}{2}+\sqrt{\frac{\tr^2(M)}{4}-1}\right),
\end{equation}
or
$$|\tr(M)|=2\cosh\left(\frac{\ell(\A)}{2}\right)
=2\cosh\left(\frac{\pi^2}{\log \rho(\A)}\right).$$
\begin{lemma}\label{lemma1}
The absolute value of the trace
of any non-parabolic element of $\Gamma(2)$
is at least~$6$.
\end{lemma}

Indeed, it is well-known and easy to prove that traces of elements
of $\Gamma(2)$ have residue $2$ modulo~$4$.

\begin{proof}[Proof of Theorem \ref{A_0}]
Let $f:\A\to \D$ be a holomorphic function in a ring $\A$, 
let $z\mapsto\lambda z$ be the fractional-linear transformation
corresponding to the generator of the fundamental group of $\A$,
as in (\ref{1}),
and let $\gamma'\in\pi(\D)$
be the $f_*$-image
of this generator.
By the assumption of the theorem, the element of $\Gamma(2)$ corresponding
to $\gamma'$ is hyperbolic, so it is conjugate
to $z\mapsto\lambda'z$ for some $\lambda'>1$.
Let $\Gamma'$ be the group generated by $z\mapsto\lambda' z$,
and consider the ring $\A'=\H/\Gamma'$. Then $f$ induces a holomorphic map
$\A\to \A'$, and we conclude from the Schwarz Lemma that
$\ell(\A')\leq\ell(\A)$, where equality holds if and only if this
holomorphic map is an isometry. So
$\lambda\geq\lambda'$.

Lemma~\ref{lemma1}  says that the trace of the matrix
$$M=\left(\begin{array}{cc}
\sqrt{\lambda^\prime}&0\\
0&\sqrt{1/\lambda'}
\end{array}\right)$$
representing
$z\mapsto\lambda' z$ is at least~$6$, which means
that $\lambda'\geq(3+2\sqrt{2})^2$. Combining this with (\ref{1}) or (\ref{4})
we obtain the inequality stated in the theorem.

To construct an extremal function, we take $\phi\in\Gamma(2)$
with $|\tr\phi|=6$, for example, $\phi=A^2B$,
and consider the ring $\H/\langle\phi\rangle$.
If $\psi$ is a conformal map of this ring onto a ring
of the form (\ref{0}) then $\Lambda\circ\psi^{-1}$  is the extremal
function.
\end{proof}

\section{Traces of hyperbolic elements of the principal
congruence subgroup}\label{fine}

In this section we prove Theorem~\ref{A_0(N_0,N_1)}.
It is deduced from (\ref{4}) and Theorem~\ref{wordlength} below
which estimates the trace of an element of $\Gamma(2)$ corresponding
to a curve with given index about $0$ and~$1$.

Consider a finite set, which we call an alphabet. A {\em cyclic word}
is a cyclically ordered finite sequence of the elements of the alphabet.
It is helpful to imagine a cyclic word as an inscription on the surface
of a cylindrical bracelet. Our alphabet consists of four letters
$$
A,B,A^{-1},B^{-1},
$$
where $A$ and $A^{-1}$ never occur next to each other, and the same about~$B$.
We use the usual abbreviation $A^m=A\ldots A$ for $m>0$ and 
$A^m=A^{-1}\ldots A^{-1}$ for $m<0$. When writing in a line,
a cyclic word can be broken at any place, so every cyclic word
distinct from $A^m$ and $B^n$ can be written as a sequence that begins
with some power of $A$ and ends with some power of~$B$.
If we substitute to such a sequence the free generators of $\Gamma(2)$,
$A,B$ as in (\ref{AB}),
then the trace of the resulting matrix depends only on the cyclic
word represented
by our sequence, because $\tr(XY)=\tr(YX)$.

So we consider a word
\begin{equation}\label{word}
w=A^{m_1}B^{n_1}\ldots A^{m_k}B^{n_k},\quad\mbox{where}\quad|m_j|\geq 1,\;
|n_j|\geq 1,
\end{equation}
satisfying 
\begin{equation}\label{word2}
\sum_{j=1}^k m_j=N_0,\quad\mbox{and}\quad\sum_{j=1}^kn_j=N_1,
\end{equation}
and estimate $|\tr w|$ from below over all
such words. Put
$$\length w =\sum_{j=1}^{k}(|m_j|+|n_j|).$$

\begin{theorem}\label{wordlength}
Let $w$ be a matrix of the form $(\ref{word})$
with $|\tr(w)|\neq 2$.
Then
\begin{equation}\label{thm6a}
|\tr w|\geq 2\length w.
\end{equation}
Moreover
\begin{equation}\label{thm6b}
|\tr w|\geq 4\length w-6\quad\mbox{if}\ \length w\ \mbox{is even}.
\end{equation}
\end{theorem}

This estimate is exact as the following words show:
$w=A(AB)^{(L-1)/2}$ with $\tr(w)=2L,$ 
when $L$ is odd,
and
$w=AB^{L-1},$ with $\tr(w)=4L-6$, when $L$ is even.
\begin{proof}[Proof of Theorem \ref{A_0(N_0,N_1)}]
Let $f\in F_0(N_0,N_1)$.
We proceed as in the proof of Theorem~\ref{A_0}.
With $\A$  and $\gamma'$ as defined there we have 
$\rho(f)=\rho(\A)$ and the word $w$ associated to 
the element of $\Gamma(2)$ corresponding to~$\gamma'$
satisfies \eqref{word} and \eqref{word2}.
The conclusion now follows from 
(\ref{4}) and Theorem~\ref{wordlength}.
\end{proof}

Another result estimating the trace from below is the following theorem
which is a special case of a result of Baribaud~\cite{Baribaud}.
\begin{theoremA}
Suppose that a closed
geodesic $\gamma$
in $\C\backslash\{0,1\}$ intersects the real line $2n$ times.
Then the trace of the element of $\Gamma(2)$ corresponding to this geodesic
is at least $2n-2$. 
\end{theoremA}
If $f\in F_0(N_0,N_1)$, then $\gamma_f$ intersects
the interval $(-\infty,0)$ at least $N_0$ times, $(0,1)$ at least
$|N_0-N_1|$ times and $(1,\infty)$ at least $N_1$ times. 
Thus $\gamma_f$ intersects the real line at least $n$ times, where
$n =N_0+N_1+|N_0-N_1|= 2 \max\{N_0,N_1\}$.
Theorem~A implies that the conclusion
of Theorem~\ref{A_0(N_0,N_1)}
 holds with $N^* = 2 \max\{N_0,N_1\} - 1$. This improves
Theorem~\ref{A_0(N_0,N_1)}
and~\eqref{10}
if $N_0+N_1$ is odd and $|N_0-N_1|>1$, but it does
not seem to be possible to deduce the conclusion of Theorem~\ref{A_0(N_0,N_1)}
in the case that $N_0+N_1$ is even.

An estimate of the trace from below for the elements of the full
modular group $\Gamma$ in terms of the word length is given in~\cite{Fine}.
It does not seem to imply Theorems~\ref{wordlength} or~A. 
Unlike the proof in~\cite{Baribaud}, which is geometric, our proof
of Theorem~\ref{wordlength} is purely algebraic, in the same spirit as
the proof in~\cite{Fine}.

\medskip

For the proof of  Theorem \ref{wordlength} we need two lemmas.
Let
$$X=\left(\begin{array}{cc}a&b\\ c&d\end{array}\right)$$
be a real matrix. We say that $X$ is {\em decreasing} if
$|a|>|b|>|d|$ and $|a|>|c|>|d|$, and we define
$$\tau(X)=|a|-|b|-|c|+|d|.$$

\begin{lemma}\label{lemma2}
Let $X\in\Gamma(2)$ be a decreasing matrix. Then
$\tau(X)\geq 0$. If $\tau(X)=0$, then $X$ has the form
$$X=\left(\begin{array}{cc}2k+1&\pm2k\\ \mp2k&-(2k-1)\end{array}\right)\quad
\mbox{or}\quad
X=-\left(\begin{array}{cc}2k+1&\pm2k\\ \mp2k&-(2k-1)\end{array}\right)$$
for some $k\in\N$. In particular, $\tau(X)=0$ implies that
$|\tr X|=2$.
\end{lemma}

\begin{proof}
Put $|a|=|b|+s$ and $|c|=|d|+t$.
Then $$\pm1=|a|\cdot|d|-|b|\cdot|c|=|d|s-|b|t,$$
and thus $|d|s\geq|b|t-1$. If $s\leq t-1$, we obtain
$$|b|t-1\leq |d|s\leq|d|(t-1)=|d|t-|d|\leq |d|t-1<|b|t-1,$$
a contradiction. Thus $s\geq t$ and hence
$$|a|+|d|=|b|+|c|+s-t\geq|b|+|c|.$$ If we have equality here, then $s=t$,
and thus
$$\pm1=|d|s-|b|t=(|d|-|b|)s.$$
Thus $t=s=1$ and $|b|=|d|+1$. It follows that there exists $k\in\N$ such that
$|a|=2k+1$, $|b|=|c|=2k$ and $|d|=2k-1$. Noting that
$ad-bc=1$, we see that $a$ and $d$ have opposite signs, and hence $X$ is of
the form given in the lemma.
\end{proof}
\begin{lemma}\label{lemma3}
Let
$$X=\left(\begin{array}{cc}a&b\\ c&d\end{array}\right)\in\Gamma(2)$$
be decreasing, $m,n\in\Z\backslash\{0\},$ and let
$$Y=A^mB^nX=\left(\begin{array}{cc}1-4mn& 2m\\ -2n&1\end{array}\right)X.$$
Then $Y$ is decreasing,
\begin{equation}\label{1a}
\tau(Y)\geq(4|mn|-2|n|-1)\tau(X)\geq\tau(X),
\end{equation}
and
\begin{equation}\label{1b}
|\tr Y|\geq|\tr X|+\tau(X)(|m|+|n|).
\end{equation}
If in addition $ad>0$, then
\begin{equation}\label{1c}
|\tr Y|\geq|\tr X|+(\tau(X)+2)(|m|+|n|).
\end{equation}
If $ad>0$ and $mn>0$, and the elements of the main diagonal of $Y$ have opposite
signs, then
\begin{equation}\label{1c1}
\tau(Y)\geq\tau(X)+2.
\end{equation}
If $mn\neq 1$, then
\begin{equation}\label{1d}
\tau(Y)\geq 3\tau(X),
\end{equation}
and if $mn\neq 1$ and $|\tr X|\neq 2$, then
\begin{equation}\label{1e}
|\tr Y|\geq 2|\tr X|+(\tau(X)+2)(|m|+|n|).
\end{equation}
\end{lemma}
\begin{proof}
Put
$$Y=\left(\begin{array}{cc}\alpha&\beta\\ \gamma&\delta\end{array}\right),$$
so that
\begin{eqnarray*}\alpha&=&(1-4mn)a+2mc,\\
                 \beta&=&(1-4mn)b+2md,\\
                 \gamma&=&-2na+c,\\
                 \delta&=&-2nb+d.\end{eqnarray*}
We may assume without loss of generality that $a>0$. Put $\mu=|m|$
and $\nu=|n|$.
We distinguish two cases.

\smallskip 
\noindent
{\em Case $1$}: $mn>0$; that is, $m$ and $n$ have the same sign.
Then
\begin{eqnarray*}
|\alpha|&=&-\alpha=(4\mu\nu-1)a-2\mu|c|\sgn(mc),\\
|\beta|&=&(4\mu\nu-1)|b|-2\mu|d|\sgn(mbd),\\
|\gamma|&=&2\nu a-|c|\sgn(nc),\\
|\delta|&=&2\nu b-|d|\sgn(nbd).\end{eqnarray*}
If $d<0$, then $b$ and $c$ must have opposite signs since $a>0$
and $ad-bc=1$. Since $m$ and $n$ have the same sign by assumption, we obtain
\begin{equation}\label{1f}
\sgn(mc)=\sgn(nc)=\sgn(nbd)=\sgn(mbd).
\end{equation}
If $d>0$, then $b$ and $c$ have the same sign. Again we conclude that
(\ref{1f}) holds. With $\varepsilon=\sgn(mc)$ we then find in both cases that
\begin{eqnarray*}
|\alpha|&=&(4\mu\nu-1)a-2\varepsilon\mu|c|,\\
|\beta|&=&(4\mu\nu-1)|b|-2\varepsilon\mu|d|,\\
|\gamma|&=&2\nu a-\varepsilon|c|,\\
|\delta|&=&2\nu b-\varepsilon|d|.
\end{eqnarray*}
Noting that, by Lemma~\ref{lemma2},
$$a-|b|=|c|-|d|+\tau(X)\geq|c|-|d|,$$
we find that
\begin{eqnarray*}
|\alpha|-|\beta|&=&(4\mu\nu-1)(a-|b|)-2\varepsilon\mu(|c|-|d|)\\
&\geq&(4\mu\nu-1-2\varepsilon\mu)(|c|-|d|)\\
&\geq&4\mu\nu-1-2\mu\geq 1,
\end{eqnarray*}
and
$$|\gamma|-|\delta|=2\nu(a-|b|)-\varepsilon(|c|-|d|)\geq
(2\nu-\varepsilon)(|c|-|d|)\geq 2\nu-1\geq 1.$$
Since $4\mu\nu\geq2\mu+2\nu$, we also have
$$|\alpha|-|\gamma|=(4\mu\nu-1-2\nu)a-\varepsilon(2\mu-1)|c|\geq
(2\mu-1)(a-|c|)\geq 1,$$
and analogously
$$|\beta|-|\gamma|\geq(4\mu\nu-1-2\nu)|b|-\varepsilon(2\mu-1)|d|\geq 1.$$
Thus $Y$ is decreasing.

Moreover,
\begin{eqnarray}
\tau(Y)&=&|\alpha|-|\beta|-(|\gamma|-|\delta|)\nonumber\\
&=&\nonumber (4\mu\nu-1-2\nu)(a+|b|)-\varepsilon(2\mu-1)(|c|-|d|)\\
&=&(4\mu\nu-1-2\nu)\tau(X) \label{1f1}\\
& &+(4\mu\nu-1-2\nu-\varepsilon(2\mu-1))(|c|-|d|)\nonumber\\
&\geq&\nonumber (4\mu\nu-1-2\nu)\tau(X)+(4\mu\nu-2\nu-2\mu)(|c|-|d|)\nonumber\\
&\geq& (4\mu\nu-1-2\nu)\tau(X)\geq\tau(X).\nonumber
\end{eqnarray}
This is (\ref{1a}).

Next we show that,
under the hypothesis stated in the lemma,
 this lower bound  for $\tau(Y)$
can be improved to \eqref{1c1} and \eqref{1d}.
We first note that if $mn\neq 1$, 
then $\mu\geq 2$ or $\nu\geq 2$ since $m$ and $n$ have the same sign.
Thus (\ref{1d}) follows from \eqref{1f1} if $mn\neq 1$.

In order to deal with \eqref{1c1} we can assume that
$ad>0$ and thus $d>0$.
Now~(\ref{1f}) yields
$$\delta=-2\nu|b|\sgn(nb)+d=-2\nu|b|\varepsilon+d,$$
and since $\alpha<0$, we find that the elements $\alpha$ and $\delta$
of the main diagonal of $Y$ have opposite signs if $\varepsilon=-1$. 
Assuming that this is the case
we deduce from (\ref{1f1}) that
\begin{eqnarray*}
\tau(Y)&\geq&(4\mu\nu-1-2\nu)\tau(X)+(4\mu\nu-2-2\nu+2\mu)(|c|-|d|)\\
&\geq&\tau(X)+2(|c|-|d|)\geq\tau(X)+2.
\end{eqnarray*}
This is (\ref{1c1}).
Hence  we have proved all claims about $\tau(Y)$ in Case~1.

We now turn to the estimates of 
$\tr Y$.
Here we distinguish two subcases.

\smallskip 
\noindent
{\em Subcase $1.1$}: $d<0$. 
Then
$\tau(X)=a-d-|b|-|c|$. Thus 
\begin{eqnarray}
-\tr Y&=&(4mn-1)a-2mc+2nb-d\nonumber\\
&\geq&(4mn-1)a -d -2\mu|c|-2\nu|b|\nonumber\\
&=&\tr X+2\tau(X)+4(\mu\nu-1)a\label{1g}
-2(\mu-1)|c|-2(\nu-1)|b|
\\
&=&\tr X+2\tau(X)+4(\mu-1)(\nu-1)a\nonumber\\
&&+2(\mu-1)(2a-|c|)+2(\nu-1)(2a-|b|).\nonumber
\end{eqnarray}
Now $2a-|c|=\tr X+\tau(X)+|b|,$ and
$2a-|b|=\tr(X)+\tau(X)+|c|$. Substituting this and using
$4(\mu-1)(\nu-1)a\geq 0$, we obtain
\begin{eqnarray*}
-\tr Y&\geq&(1+2(\mu-1)+2(\nu-1))\tr X
+(2+2(\mu-1)+2(\nu-1))\tau(X)\\
& &+2(\mu-1)|b|+2(\nu-1)|c|.
\end{eqnarray*}
Since $|b|\geq 2$ and $|c|\geq 2$, this yields
\begin{equation}\label{1h}
|\tr Y|\geq(2\mu+2\nu-3)\tr X+(2\mu+2\nu-2)\tau(X)+4(\mu+\nu-1).
\end{equation}
Now (\ref{1b}) follows since $\mu\geq 1$ and $\nu\geq 1$, and
thus $2\mu+2\nu-2\geq\mu+\nu$. Moreover,
(\ref{1h}) may be written in the form
\begin{eqnarray*}
|\tr Y|&\geq&(2\mu+2\nu-3)\tr X+(\mu+\nu)(\tau(X)+2)\\
& & + (\mu+\nu-2)(\tau(X)+2)-4.
\end{eqnarray*}
By Lemma~\ref{lemma2}, we have $\tau(X)\geq 2$ if $|\tr X|\neq 2$.
In $mn\neq 1$ and hence $\mu\geq 2$ or $\nu\geq 2$, we have
$(\mu+\nu-2)(\tau(X)+2)\geq\tau(X)+2\geq 4,$ and (\ref{1e}) follows.

\smallskip 
\noindent
{\em Subcase $1.2$}: $d>0$.
Since $ad-bc=1$,
we see that $b$ and $c$ have the same sign. Thus $mc$ and
$nd$ have the same sign. We may assume that $mc>0$.
The case that $mc<0$ is analogous. We have similarly to (\ref{1g})
\begin{eqnarray}
-\tr Y&=&(4mn-1)a-2mc+2nb-d\nonumber\\
&=&(4\mu\nu-1)a-2\mu|c|+2\nu|b|-d\nonumber\\
&=&\tr X +2\tau(X)+4(\mu\nu-1)a
\nonumber\\ & &
-2(\mu-1)|c|+2(\nu-1)|b|+4|b|-4d\label{1i}\\
&=&\tr X+2\tau(X)+4(\mu-1)(\nu-1)a\nonumber\\
& &+2(\mu-1)(2a-|c|)+2(\nu-1)(2a+|b|)+4(|b|-d)\nonumber\\
&\geq&\tr X+2\tau(X)+4(\mu-1)(\nu-1)a\nonumber\\
& &+2(\mu-1)(2a-|c|)+2(\nu-1)(2a-|b|).\nonumber
\end{eqnarray}
Now
$$2a-|c|=a+\tau(X)+|b|-d\geq\frac{1}{2}\tr X+\tau(X),$$
and 
$$2a+|b|\geq\frac{1}{2}\tr(X)+\tau(X).$$
Substituting this into (\ref{1i}) yields
\begin{eqnarray*}
|\tr Y|&\geq&(1+\mu-1+\nu-1)\tr X
+(2+2(\mu-1)+2(\nu-1))(\tau(X)+2)\\
&\geq&(\mu+\nu-1)\tr X+(\mu+\nu)(\tau(X)+2),
\end{eqnarray*}
from which (\ref{1c}) follows.
In particular, we have \eqref{1b}. Moreover,
if $mn\neq 1$, then  $\mu+\nu-1\geq 2$ 
and  thus \eqref{1e} follows.

\smallskip 
\noindent
{\em Case $2$}: $mn<0$; that is, $m$ and $n$ have opposite signs.
Putting again $\varepsilon=\sgn(mc)$ we now find that
\begin{eqnarray*}
|\alpha|&=&(4\mu\nu+1)a+2\varepsilon\mu|c|,\\
|\beta|&=&(4\mu\nu+1)|b|+2\varepsilon\mu|d|,\\
|\gamma|&=&2\nu a+\varepsilon|c|,\\
|\delta|&=&2\nu b+\varepsilon|d|.
\end{eqnarray*}
The proof that $Y$ is decreasing is analogous to Case~1.
As in \eqref{1f1} we find that
\begin{eqnarray}
\tau(Y)
&=&\nonumber (4\mu\nu+1-2\nu)(a+|b|)+\varepsilon(2\mu-1)(|c|-|d|)\\
&=&(4\mu\nu+1-2\nu)\tau(X) \nonumber
+(4\mu\nu+1-2\nu+\varepsilon(2\mu-1))(|c|-|d|)\nonumber\\
&\geq&\nonumber (4\mu\nu+1-2\nu)\tau(X)+(4\mu\nu-2\nu-2\mu+2)(|c|-|d|)\nonumber,
\end{eqnarray}
from which the claimed lower bounds for $\tau(Y)$
easily follow.

Moreover,
\begin{eqnarray*}
\tr Y &=&(4\mu\nu+1)a+2mc-2nb+d\\
&\geq&\tr X+4\mu\nu a-2\mu|c|-2\nu|b|\\
&=&\tr X +2\tau(X)+4(\mu-1)(\nu-1)a\\
& & +2(\mu-1)(2a-|c|)+2(\nu-1)(2a-|b|)+2(a-|d|).
\end{eqnarray*}
Since $2(a-|d|)\geq 0$, we obtain the same inequalities as in (\ref{1g})
and (\ref{1i}), and the conclusion follows from this as above.
\end{proof}

\begin{proof}[Proof of Theorem \ref{wordlength}]
If  $m_j=n_j=1$ for all $j\in \{1,\ldots,k\}$ then an easy induction
shows that
$$w=(-1)^k\left(\begin{array}{cc}2k+1&-2k\\ 2k&1-2k\end{array}\right),$$
and if $m_j=n_j=-1$ for all~$j$, then 
$$w=(-1)^k\left(\begin{array}{cc}2k+1&2k\\ -2k&1-2k\end{array}\right).$$
Thus $|\tr w|=2$ in these cases. Hence we may assume that not all $m_j,n_j$
are equal to $1$ and not all of them are equal to $-1$.

Suppose first that there exists $j$ such that $|m_j|\neq 1$ or $|n_j|\neq 1$.
Since the trace of a product does not change under cyclic permutations,
we may assume that $|m_k|\neq 1$ ot $|n_k|\neq 1$.
Then
$$|\tr(A^{m_k}B^{n_k})|=|2-4m_kn_k|
\geq 4 |m_kn_k|-2\geq 2(|m_k|+|n_k|)\geq 6.$$
Now Lemma~\ref{lemma2} implies that
$$\tau(A^{m_k}B^{n_k})\geq 2.$$
It follows now from Lemma~\ref{lemma3} and (\ref{1b}) that
\begin{eqnarray*}
|\tr(A^{m_{k-1}}B^{n_{k-1}}A^{m_k}B^{n_k})|
&\geq&|\tr(A^{m_k}B^{n_k})|+2 (|m_{k-1}|+|n_{k-1}|)\\
&\geq&2(|m_{k-1}|+|n_{k-1}|)+2(|m_{k}|+|n_{k}|).
\end{eqnarray*}
Now (\ref{thm6a}) follows by induction.

Next we note that
\begin{eqnarray}
|\tr(A^{m_k}B^{n_k})|&=&4|m_kn_k|-2\nonumber\\
&=&(2|m_k|-2)(2|n_k|-2)+4(|m_k|+|n_k|)-6\nonumber\label{2a}\\
&\geq&4(|m_k|+|n_k|)-6.\nonumber
\end{eqnarray}
Thus (\ref{thm6b}) follows from (\ref{1a}) and (\ref{1b}) by induction,
whenever
\begin{equation}\label{2b}
\tau(A^{m_k}B^{n_k})\geq 4.
\end{equation}
If $|m_k|\geq 3$ or $|n_k|\geq 3$, then
$$\tau(A^{m_k}B^{n_k})=4|m_kn_k|-2|m_k|-2|n_k|=(2|m_k|-1)(2|n_k|-1)-1\geq 4,$$
so that (\ref{2b}) holds. It is easy to see that (\ref{2b}) also holds
if $|m_k|=|n_k|=2$ or if $m_kn_k=-2$. Thus (\ref{thm6b}) holds
in these cases, even if the $\length w$ is odd. 

The remaining case, apart from the case that $|m_j|=|n_j|=1$ for all 
$j\in\{1,\ldots,k\}$, is the case that $m_kn_k=2$. 
Here the hypothesis that the $\length w$
is even implies that there exists $l\in\{1,\ldots,k-1\}$
such that $(m_l,n_l)\neq(1,1)$, and $(m_l,n_l)\neq(-1,1)$.
Lemma~\ref{lemma3}, (\ref{1b}) and the previous arguments now imply that
$$|\tr(A^{m_{l+1}}B^{n_{l+1}}\ldots A^{m_k}B^{n_k})|\geq 2\sum_{j=l+1}^k(|m_j|+|n_j|).$$
Moreover, by Lemma~\ref{lemma3}, (\ref{1d}) and (\ref{1e}) we obtain
$$\tau(A^{m_l}B^{n_l}\ldots A^{m_k}B^{n_k})\geq3\tau(A^{m_{l+1}}B^{n_{l+1}}
\ldots A^{m_k}B^{n_k})\geq 6,$$
and
\begin{eqnarray*}
|\tr(A^{m_l}B^{n_l}\ldots A^{m_k}B^{n_k}))
&\geq&2|\tr(A^{m_{l+1}}B^{n_{l+1}}\ldots A^{m_k}B^{n_k})|+4(|m_l|+|n_l|)\\
&\geq&4\sum_{j=l}^k(|m_j|+|n_j|).
\end{eqnarray*}
Now (\ref{thm6b}), and in fact a stronger inequality, follows from
Lemma~\ref{lemma3}, (\ref{1a}) and (\ref{1b}) by induction.

It remains to consider the case that $|m_j|=|n_j|=1$ for all
$j\in\{1,\ldots,k\}$, but not all $m_j$ and $n_j$ have the same sign.
Using cyclic permutations we may assume that $m_{k-1},n_{k-1},m_k$ and $n_k$
do not all have the same sign. If $m_k=-n_k\; (=\pm1)$, then
$$A^{m_k}B^{n_k}=A^{\pm1}B^{\mp1}=\left(\begin{array}{cc}5&\pm2\\\pm2&1\end{array}
\right),$$
and thus $$\tr(A^{m_k}B^{n_k})=6=4(|m_k|+|n_k|)-2.$$
If there exists $l\in\{1,\ldots,k-1\}$ such that $m_l=-n_l\; (=\pm1),$
the proof can be completed as before. We thus may assume that
$$(m_j,n_j)\in\{(1,1),(-1,-1)\}\quad\mbox{for}\quad 1\leq j\leq k-1.$$
Suppose that for each $j$ satisfying $l+1\leq j\leq k$ the two main
diagonal elements of
\begin{equation}\label{nww}
A^{m_{j+1}}B^{n_{j+1}}\ldots A^{m_k}B^{n_k}
\end{equation}
have the same sign (which may depend on $j$). Then by (\ref{1c})
and induction
\begin{eqnarray*}
\tr(A^{m_l}B^{m_l}\ldots A^{m_k}B^{m_k})&\geq&\tr(A^{m_k}B^{m_k})+
4\sum_{j=l}^{k-1}(|m_j|+|n_j|)
\\ &\geq&
4\sum_{j=l}^k(|m_j|+|n_j|)-2.
\end{eqnarray*}
If $l=1$, we obtain (\ref{thm6b}). If, however, the main diagonal elements of
$$A^{m_{l+1}}B^{m_{l+1}}\ldots A^{m_k}B^{m_k}$$
have opposite signs for some~$l$, and $l$ is the largest number with this
property, then by (\ref{1c1}) we have
$$\tau(A^{m_{l+1}}B^{m_{l+1}}\ldots A^{m_k}B^{m_k})\geq
\tau(A^{m_{l+2}}B^{m_{l+2}}\ldots A^{m_k}B^{m_k})+2\geq 4.$$
Now the proof is again completed using (\ref{1a}) and (\ref{1b}).

The case that $m_{k-1}=-n_{k-1}$ can be reduced to the previous one
by cyclic permutation.

Suppose finally that $m_k=n_k=-m_{k-1}=-n_{k-1}$. Then
$$A^{m_{k-1}}B^{n_{k-1}}A^{m_k}B^{n_k}=\left(\begin{array}{cc}13&\pm8\\
\pm8&5\end{array}\right),$$
so that
$$\tr(A^{m_{k-1}}B^{n_{k-1}}A^{m_k}B^{n_k})=18=4\sum_{j=k-1}^k(|m_j|+|n_j|)+2.$$
The proof is now completed by the same arguments as before, distinguishing
the cases whether the elements of the main diagonal of (\ref{nww})
have the same sign or not.
\end{proof}

The proof of Theorem~\ref{wordlength} shows that the computation 
of $A_0(N_0,N_1)$ is equivalent to minimizing $|\tr(w)|$ among
all words $w$ of the form~\eqref{word} which satisfy~\eqref{word2},
the connection being given by~\eqref{4}.

In order to find this minimum fix a word $w_0$ satisfying~\eqref{word2}.
There are only finitely many words $w_1,\dots,w_n$ satisfying~\eqref{word2}
for which $\length(w_j)\leq |\tr(w_0)|/2$.
By Theorem~\ref{wordlength} it suffices to check these words.
With $T:=\min_{0\leq j\leq n} |\tr(w_j)|$ we deduce from~\eqref{4} that
$A_0(N_0,N_1)=\exp\left(\pi^2/\cosh^{-1}(T/2)\right)$.

\section{An exact formula and
a conjecture about the traces}\label{conjectures}

This section is not needed for understanding of the rest of the paper;
it contains a formula for the traces and
a conjecture about them that are of independent interest.

We are looking at 
the trace of the matrix (\ref{word}) which we denote by
$$\left(\begin{array}{cc}a_{1,1}&a_{1,2}\\a_{2,1}&a_{2,2}\end{array}\right).$$ 
A formula for this matrix and its trace can be easily proved by induction.
We have
$$A^n=\left(\begin{array}{cc}1&2n\\0&1\end{array}\right)
\quad\text{and}\quad
B^{-n}=\left(\begin{array}{cc}1&0\\2n&1\end{array}\right).$$
The main diagonal elements $a_{1,1}$ and $a_{2,2}$
of the matrix $A^{m_1}B^{-n_1}\ldots A^{m_k}B^{-n_k}$
are 
given by
\begin{eqnarray*}
a_{1,1}
&=&1+4\sum_{i\leq j}m_in_j+16\sum_{i_1\leq j_1<i_2\leq j_2}
m_{i_1}n_{j_1}m_{i_2}n_{j_2}+\ldots\\
& & +4^\ell\sum_{i_1\leq j_1<i_2\leq j_2<\ldots\leq j_\ell}
m_{i_1}n_{j_1}m_{i_2}n_{j_2}\ldots m_{i_\ell}n_{j_\ell}+\ldots\\
& & +4^k m_1n_1\ldots m_kn_k
\end{eqnarray*}
and
\begin{eqnarray*}
a_{2,2}&=&1+4\sum_{i> j}m_in_j+16\sum_{i_1>j_1\geq i_2>j_2}
m_{i_1}n_{j_1}m_{i_2}n_{j_2}+\ldots\\
& & +4^\ell\sum_{i_1>j_1\geq i_2>j_2\geq\ldots>j_\ell}
m_{i_1}n_{j_1}m_{i_2}n_{j_2}\ldots m_{i_\ell}n_{j_\ell}+\ldots\\
& & +4^{k-1}n_1m_2n_2\ldots n_{k-1}m_k.
\end{eqnarray*}
The trace polynomial $a_{11}+a_{22}$ seems
to have the followings remarkable
property which we verified for $k\leq 6$ using Maple.

\begin{conj}
If we substitute $m_i=\pm(1+p_i)$ and $n_j=\pm(1+q_j)$
with {\em arbitrary} combination
of signs $\pm$,
then we obtain a polynomial in $p_i,q_j$ with coefficients
of {\em constant} sign. This constant sign is equal to the sign
of the monomial $\pm4^kp_{1}q_1\ldots p_kq_k$
of the highest degree $2k$.
\end{conj}

The off-diagonal elements of the product are 
given by
\begin{eqnarray*}
a_{1,2}&=&2\sum_{j=1}^k m_j+8\sum_{i_1\leq j_1<i_2}m_{i_1}n_{j_1}m_{i_2}+\ldots\\
& & +2\cdot 4^\ell\sum_{i_1\leq j_1<i_2\leq j_2<\ldots\leq j_{\ell-1}<i_\ell}
m_{i_1}n_{j_1}\ldots m_{i_\ell}+\ldots
\end{eqnarray*}
and
\begin{eqnarray*}
a_{2,1}&=&2\sum_{j=1}^k n_j+8\sum_{j_1<i_1\leq j_2}n_{j_1}m_{i_1}n_{j_2}+\ldots\\
& & +2\cdot 4^\ell\sum_{j_1<i_1\leq j_2<\ldots<i_{\ell-1}\leq j_\ell}
n_{j_1}m_{i_1}\ldots n_{j_\ell}+\ldots.
\end{eqnarray*}

\section{Proofs of Theorems \ref{AAA} and \ref{wk}}\label{cs}

In the rest of the paper we discuss $A_2,A_5$ and related constants.
The idea is that
if $f\in F_2$, then we can replace the ring $\{\rho(f)<|z|<1\}$ by a 
larger domain and obtain
a stronger estimate than Theorem~\ref{A_0(N_0,N_1)} using the same method.
\begin{proof}[Proof of Theorem \ref{AAA}]
Let  $f\in F_2(N_0,N_1)$ and denote by
$q$ the cardinality of $f^{-1}(\{0,1\})$.
Consider the union $K$ of the $q$ closed segments from $0$
to the points of $f^{-1}(\{0,1\})$.
We proceed as in the proof of Theorems \ref{A_0} and
\ref{A_0(N_0,N_1)}, but replace the ring 
$\A=\{z:\rho(f)<|z|<1\}$ considered there
by the ring $\A:=\U\backslash K$.
The arguments used in these proofs now yield that
 $\rho(\A)\geq A_0(N_0,N_1)$.

For $r>0$
we denote by  $K_q(r)$ the union of the  $q$
segments from $0$ to $re^{2\pi ij}$, $0\leq j\leq q-1$.
A result of Dubinin \cite[Lemma 1.4 and Theorem 1.10]{Dubinin} 
implies that $\mod(\U\setminus K)\geq \mod(\U\setminus K_q(\rho(f))$. 
With $r:=\rho(f)$ and 
$\A_q(r)=\U\backslash K_q(r)$ we thus have
$\rho(\A)\leq \rho(\A_q(r))$.
With $R_q(r):=\rho(\A_q(r))$ we hence find that
\begin{equation}\label{Rq}
R_q(r)\geq A_0(N_0,N_1).
\end{equation}
Since $z\mapsto z^q$ is a covering from $\A_q(r)$ onto $\A_1(r^q)$
we see that
$$q \mod \A_q(r) = \mod \A_1(r^q).$$
Using the estimate (see, for example,~\cite[section II.2.3]{LV})
$$\mod \A_1(t)\geq \frac{1}{2\pi} 
\log \frac{\left(1+\sqrt{1-t^{2}}\right)^2}{t}$$ 
we obtain
$$q \mod \A_q(r)\geq \frac{1}{2\pi} 
\log \frac{\left(1+\sqrt{1-r^{2q}}\right)^{2}}{r^q}.$$
Using \eqref{0} we thus find that 
$$
R_q^q(r)=\rho(\A_q(r))^q = \exp(-2\pi q\mod \A_q(r))
\leq \frac{r^q}{\left(1+\sqrt{1-r^{2q}}\right)^{2}}.$$
If $r\leq 2^{2/q}R_q(r)$ so that $r^{2q}\leq 16 R_q^{2q}(r)$, this implies that
\begin{equation}\label{triv}
\rho(f)=r\geq R_q(r)\left(1+\sqrt{1-16R_q^{2q}(r)}\right)^{2/q},
\end{equation}
and if $r>2^{2/q}R_q(r)$, then (\ref{triv}) is trivially satisfied.
Combining \eqref{Rq} and \eqref{triv} we obtain~\eqref{AAA1}.

To prove~\eqref{AAA2} we note that Theorem~\ref{A_0(N_0,N_1)}
and the subsequent remarks imply that if $N_0+N_1>3$,
then 
$$A_2(N_0,N_1)\geq A_0(N_0,N_1)\geq A_0(3,1)\approx 0.013968
>0.00587465.$$
Thus it suffices to consider functions $f\in F_2(2,1)$.
For these functions we have $q\in\{2,3\}$.
We insert these values for $q$ and $(N_0,N_1)=(2,1)$  into~\eqref{AAA1}
and find that 
the smaller bound is obtained for $q=3$.
This yields~\eqref{AAA2}.
\end{proof}

We derive two corollaries from Theorem~\ref{AAA}.
For $f\in F_2$ we 
considered the curve 
$\gamma_f=f(\{ z:|z|=\sqrt{\rho(f)}\})$.
To the curve $\gamma_f$ corresponds a cyclic word
in the alphabet $A,B,A^{-1},B^{-1}$, and we denote this cyclic
word by
$w(f)$.

\begin{corollary} \label{cor1}
Let $f$ be an extremal function for $A_2$.
Then the cyclic word $w(f)$ can be one of the following:
\begin{equation}\label{words}
A^2B,\quad A^3B,\quad A^4B,\quad A^2BAB,\quad A^2BABAB,
\end{equation} or a word obtained from one of these
by permutation of $A$ and~$B$.
\end{corollary}
\begin{proof}
We know from Goldberg's result that $A_2\leq 0.032$,
on the other hand, if
$|\tr w(f)|\geq 18$, then
by formula (\ref{4}) we have $\rho(f)\geq 0.0327$.
So for the extremal function $f$ we must have $|\tr w(f)|\leq 14$.
Using Theorem~\ref{wordlength}, one can easily make a complete
list of words $w$ with $|\tr w|\leq 14$. Up to cyclic permutation
or replacement of $(A,B)$ by $(B,A)$, $(A^{-1},B^{-1})$ or
$(B^{-1},A^{-1})$,
these words are~(\ref{words}).
\end{proof}

If we are willing to use the
numerical value of $\mu$ from Theorem~\ref{A_2}
instead of the Goldberg estimate,
then Corollary~\ref{cor1} can be strengthened:
\begin{corollary} \label{cor2} 
{\rm (Computer assisted)}
Let $f$ be an extremal function for $A_2$.
Then the cyclic word $w(f)$ can be one of the following:
$$
A^2B,\quad A^3B,\quad A^2BAB,
$$
or a word obtained from one of these
by permutation of $A$ and~$B$.
\end{corollary}
\begin{proof}
If $w(f)=A^4B$, then $f^{-1}(\{0,1\})$ contains at most $5$ points.
Thus Theorem~\ref{AAA} and  (\ref{14}) give
$$\rho(f)\geq\left(1+\sqrt{1-16A_0(4,1)}\right)^{2/5}A_0(4,1)>0.0310>
\mu\approx0.252896\geq A_2,$$
which contradicts our assumption that $f$ is extremal for $A_2$. 

If $w(f)=A^2BABAB$, then $f^{-1}(\{0,1\})$
contains at most $7$ points,
and Theorem~\ref{AAA} with $A_0(4,3)= A_0(4,1)\approx0.0235$ gives
$\rho(f)>0.286>\mu\geq A_2$,
which also contradicts the assumption of extremality
of~$f$. Thus only three words remain as stated in the corollary.
\end{proof}

In the proofs of Theorems \ref{A_0}-\ref{AAA} we considered 
the hyperbolic length  $\ell(\A)$ of the  shortest geodesic 
separating the two boundary components of a ring~$\A$.
In the following, we shall consider
domains of higher multiplicity.
For a compact, not necessarily connected
subset $K$ of the unit disk $\U$ we denote by $\ell(\U\setminus K)$ 
the hyperbolic length of 
the  shortest geodesic separating $K$ and $\partial \U$
in $\U\setminus K$.

Note that $h$ is a covering which maps the geodesics separating
$\{-\mu,\mu\}$ from $\partial \U$ in $\U\setminus \{-\mu,\mu\}$
to the geodesic in $\D=\C\setminus\{0,1\}$ which is of class $A^2B$.
The latter geodesic is shown in Figure~\ref{fig1}, right.
Its length is $2\log(3+2\sqrt{2})$ by Theorem~${1.1^\prime}$  
or~\eqref{minlength}.
Thus
\begin{equation}\label{ellmu}
\ell(\U\setminus \{-\mu,\mu\})=2\log(3+2\sqrt{2}).
\end{equation}
\begin{proof}[Proof of Theorem~\ref{wk}]
Let $f\in F_5(m,n)$ with $m>n\geq 1$.
First we note that
$\ell(\U\setminus f^{-1}(\{0,1\}))$ remains unchanged if
$f$ is replaced by $f\circ T$ for some $T\in\Aut(\U)$,
but $\rho(f\circ T)$ is minimal if the zero of $f\circ T$
is the negative of the $1$-point of $f\circ T$.
With $r=\rho(f)$ we may thus assume that $f(-r)=0$
and $f(r)=1$.
Let $\gamma$ be the geodesic
separating $\{-r,r\}$ from $\partial \U$ in
$\U\setminus \{-r,r\}$.
Then $f(\gamma)$ is an non-peripheral curve in $\D$.
By Theorem~$1.1'$ the hyperbolic length of $f(\gamma)$ in
$\D$ is at least $2\log(3+2\sqrt{2})$.
By the Schwarz Lemma, the hyperbolic length of $\gamma$ in
$\U\setminus \{-r,r\}$ has the same lower bound and thus
\begin{equation}\label{ellr}
\ell(\U\setminus\{-r,r\}\geq 2\log(3+2\sqrt{2}).
\end{equation}
Next we note that 
$\ell(\U\setminus \{-t,t\})$ is an increasing function of~$t$.
In fact, $\U\setminus \{-t,t\}$ is conformally equivalent to 
$\{z:|z|<s/t\}\setminus \{-s,s\}$ 
and 
$$\{z:|z|<s/t\}\setminus \{-s,s\}\subset \U\setminus \{-s,s\}$$
for $0<s<t<1$. Since the hyperbolic metric increases if the domain
decreases, we see conclude that $\ell(\U\setminus  \{-s,s\})
\leq \ell(\U\setminus \{-t,t\})$ for $0<s<t<1$.
From~\eqref{ellmu} and~\eqref{ellr}
 we now deduce that $\rho(f)=r\geq \mu$.

If we have equality, then $f$ must be a covering 
and the hyperbolic length of $f(\gamma)$ in
$\D$ must be equal to $2\log(3+2\sqrt{2})$.
This implies that the trace of the word associated to
the curve $f(\gamma)$ is equal to~$6$. Theorem~\ref{wordlength},
together with the assumption that $m>n$, now yields that this 
word must be $A^2B$. We deduce that $f=h$.
\end{proof}

\section{Locally extremal functions.} \label{h}

A function $g\in F_2$ is called {\em locally extremal}
if
$$g:\U\backslash g^{-1}(\{0,1\})\to\D=\C\backslash\{0,1\}$$
is a covering map.

Locally extremal functions are labeled by subgroups
$$\Gamma(g):=g_*(\pi(z_0,\U\backslash g^{-1}(\{0,1\})\subset\Gamma(2).$$
These subgroups are generated by finitely many elements,
each of which represents a counterclockwise loop, possibly multiple, around
$0$ or~$1$. We recall that each subgroup $\Gamma$ of the fundamental group
$\pi(w_0,\D)$ corresponds to a covering $g:(X,x_0)\to (\D,w_0)$
such that $\Gamma=\Gamma(g)=g_*(\pi(x_0,X))$, where $X$
is a hyperbolic Riemann surface which is
unique up to conformal equivalence;
 see, for example,~\cite[Section~9.4]{Ahlfors2}.

In general, parabolic elements in the fundamental group of a Riemann
surface correspond to loops around punctures in the surface.
If the fundamental group of a Riemann surface 
is generated by finitely many parabolic elements, then the Riemann 
surface is conformally equivalent to the plane with finitely many 
punctures or a disk with finitely many punctures.
The first possibility occurs if and only if the product 
of the parabolic elements generating the fundamental group is also 
parabolic.

We are interested in the case that  $\Gamma=\langle A^m,B^n\rangle$ 
where $m,n\in\N$, $m\neq n$. We have $\tr (A^m B^n)=4mn-2\geq 6$
and hence $A^m B^n$ is not parabolic. Thus we may take $X$ to 
be the unit disk with two punctures, which we 
can place at the points $-\mu_{m,n}$ and $\mu_{m,n}$ for some $\mu_{m,n}>0$. 
This determines $\mu_{m,n}$ uniquely.
Moreover, $g$ is defined uniquely up to precomposition by 
$z\mapsto -z$. The functions $g$ extend to the unit disk, 
taking the values $0$ and $1$ at the punctures $-\mu_{m,n}$ and $\mu_{m,n}$,
and we define $h_{m,n}$ to be the function which takes the value
$0$ at $-\mu_{m,n}$.
The constant $\mu$ and the function $h$ mentioned in the introduction
are given by $\mu=\mu_{2,1}$ and $h=h_{2,1}$.
\begin{theorem}\label{lemma4}
Let $f\in A_5(m,n)$.
Then $f$ is subordinate to 
$h_{m,n}$.
\end{theorem}

\begin{proof}
Let $z_0,z_1\in\U$ such that $f(z_0)=0$ and $f(z_1)=1$.
Choose $z^*\in \U\setminus\{z_0,z_1\}$. Then $f(z^*)\not\in\{0,1\}$.
Consider simple loops $\gamma_j$ in $\U\backslash\{ z_0,z_1\}$,
beginning and ending at $z^*$ and going once around $z_j$ counterclockwise. 
The images $f(\gamma_j)$ are curves in $\D$
which begin and end at $f(z^*)$.
They represent elements $\delta_0$ and $\delta_1$
 of the fundamental group
$\pi(f(z^*),\D)$ which correspond to $A^m$ and $B^n$.
Indeed, the $\gamma_j$ are freely homotopic to small simple loops around
$z_j$, so the $\delta_j$ are 
freely homotopic to loops  of multiplicity $m$ and $n$
around $0$ and~$1$.
 
For some germ $\varphi$ of the inverse function of 
$h_{m,n}$ we now consider the function $\omega=\varphi\circ f$.
As $h_{m,n}:\U\setminus\{-\mu_{m,n},\mu_{m,n}\}\to \Omega$ 
is a covering, $\varphi$ can be continued analytically
along every curve in $\U\backslash\{ z_0,z_1\}$.
Moreover, it follows from the above consideration that
the monodromy around the punctures $z_0$ and $z_1$ is trivial.
Thus $\omega$ extends to a map $\omega:\U\to\U$ with
$\omega(z_0)=-\mu_{m,n}$ and $\omega(z_1)=\mu_{m,n}$.
Clearly, $h_{m,n}\circ \omega= f$.
\end{proof}

In the next two sections we compute $\mu=\mu_{2,1}$
numerically.

\section{Explicit construction of the coverings}\label{a_2}
We express the local extremal function corresponding
to the subgroup $\langle A^2,B\rangle$
in terms of special conformal mappings. 
This function is the 
extremal function $h=h_{2,1}$ described in the introduction.

Let $\H$ be the upper half-plane and consider the regions (cf.\
Figure~\ref{figG})
$$
G=\left\{ z\in \H:0<\Rea z|<2,
\left| z-\frac{1}{2}\right|>\frac{1}{2}\right\}
$$
and
$$G'=\left\{ z\in \H:0<\Rea z<1,
\left| z-\frac{1}{2}\right|>\frac{1}{2}\right\}.$$
\begin{figure}[htb]
\begin{center}
\includegraphics[height=5cm]{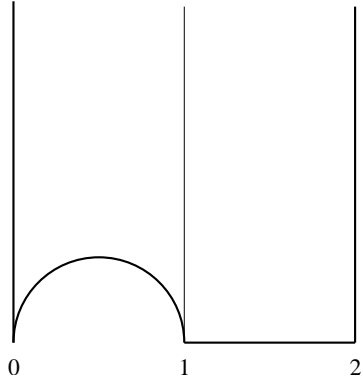}
\caption{The regions $G$ and $G'$.}
\label{figG}
\end{center}
\end{figure}

Let $\phi:\H\to G$ and $\Lambda:G'\to \H$
be conformal homeomorphisms with
the following boundary correspondence:
$$\phi:(-1,1,\infty)\mapsto (1,2,\infty),
\quad\Lambda:(0,1,\infty)\mapsto(1,\infty,0).$$
We define 
\begin{equation}
\label{defa}
-a=\phi^{-1}(0)<-1.
\end{equation}
Then $\Lambda$ extends to the upper
half-plane by reflections, so the composition 
\begin{equation}\label{tau}
\tau=\Lambda\circ\phi 
\end{equation}
is a well defined analytic function in~$\H$. Evidently, it omits $0,1,\infty$,
and $\tau'(z)\neq 0$ in~$\H$. 
Now it is not difficult to
check that the boundary values
of $\tau$ map the interval $\R\cup\{\infty\}\backslash[-1,1]$
into $\R\cup\{\infty\}$.
Hence, by the Schwarz Reflection Principle,
$\tau$ is meromorphic in $\bC\backslash[-1,1]$.
It is equally easy to check that $\tau$ has a double zero at $\infty$ and
a simple $1$-point at $-a$.

The Joukowski function $J(z)=(z+z^{-1})/2$ maps the unit disk conformally onto
$\bC\backslash[-1,1]$, with $J(0)=\infty$ and $J(q)=-a$,
where
\begin{equation}
\label{q}
q=-a+\sqrt{a^2-1}\in(-1,0).
\end{equation}
Now we define the real conformal automorphism
\begin{equation}\label{chi}
\chi(z)=\frac{z-\mu}{1-\mu z}
\end{equation}
of the unit disk that
sends $(-\mu,\mu)$ to $(q,0)$. We obtain 
\begin{equation}\label{mu}
\mu=\frac{-1+\sqrt{1-q^2}}{q}.
\end{equation}
Then we define our function
\begin{equation}\label{defh}
h=1-\tau\circ J\circ\chi.
\end{equation}
All properties (a)--(d) of $h$ from the introduction are evident now.
These properties determine our function $h$
uniquely. 

The union of~$G$, its reflection in the imaginary axis and the
positive imaginary axis is a fundamental domain of the group generated
by $A^2$ and $B$: $A$ does the vertical sides pairing,
and $B$ pairs
the circles. So $\Gamma(h)=\langle A^2,B\rangle$.

\section{Computation of the constant $\mu$}\label{A}

We compute the value $a$ in (\ref{defa})
We use the notation from the previous section.
The function $\phi$ extends by symmetry to 
\begin{equation}\label{formula}
\phi:Q_1=\H\cup\overline{\H}\cup(-1,1)\to Q_2= G\cup\overline{G}\cup(1,2).
\end{equation}
This can be considered as a conformal map between two quadrilaterals.
The first quadrilateral has two vertices at $\infty$ and two at~$a$,
the second one two vertices at $\infty$ and two at~$0$.

Every quadrilateral with a chosen pair
of opposite sides can be mapped conformally onto a rectangle, so
that the chosen sides go to the vertical sides.
Such a map is unique, up to rotation of the rectangle by~$\pi$.
The preimage of the center of this rectangle
will be called the {\em center} of the quadrilateral.

The harmonic measure of one vertical side at the center is a conformal
invariant of a quadrilateral.
Our strategy is to 
compute 
the harmonic measure $\omega_0$
 of the circle $|z-1/2|=1/2$ at the center of the 
quadrilateral $Q_2$ in (\ref{formula}) numerically. The harmonic measure
of the corresponding side $[-a,-1]$ of the quadrilateral $Q_1$
can be explicitly computed in terms of~$a$.
Comparison of the harmonic measures at the centers will give the value of~$a$.

Now we give the details.
First we handle $Q_1$. The part of the boundary that corresponds to
the circle of $\partial Q_2$ is the interval $[-a,-1]$. First we map $Q_1$
onto the unit disk  $\U$ by the composition of the real maps
$$z_1=\sqrt{\frac{1+z}{1-z}},\quad\mbox{where}\ \sqrt{w}>0\ \mbox{for}\  w>0,$$
and
$$z_2=\frac{z_1-1}{z_1+1}.$$
The points $-a^+$ and $-a^-$ on the upper and lower sides of $(-\infty,-1]$
are mapped to
\begin{equation}\label{ba}
b=\frac{-1+i\sqrt{a^2-1}}{a}\quad\mbox{and}\quad \overline{b},
\end{equation}
respectively.
We have
$$\Rea b=-\frac{1}{a}\quad\mbox{and}\quad\Ima b=\sqrt{1-\frac{1}{a^2}}.$$
The points $\infty^+$ and $\infty^-$ are mapped to $i$ and
$-i$, respectively.
It is not difficult to see that there exists a real automorphism
$\phi$ of the unit disk $\U$
and $w\in\partial\U$ lying in the first quadrant
which realize  the following boundary correspondence:
$$\phi: (i,b,\overline{b},-i)\mapsto(w,-\overline{w},-w,\overline{w})=:
(w,w_1,w_2,w_3).$$
Since the cross-ratio is invariant under fractional-linear
transformations, we find that
$$\left|\frac{b-i}{b+i}\right|^2
=\frac{(i-b)(-i-\overline{b})}{(i-\overline{b})(-i-b)}
=\frac{(w-w_1)(w_3-w_2)}{(w-w_2)(w_3-w_1)}
=(\Rea w)^2.$$
For the harmonic measure $\omega_0$
of a ``vertical side'' at the center which we are
searching we  thus have
$$\omega_0=\frac{1}{\pi}\arccos(\Rea w)=\frac{1}{\pi}\arccos
\left(\left|\frac{b-i}{b+i}\right|\right).$$
Using~\eqref{ba} we obtain
$$
\omega_0=\frac{1}{\pi}\arccos(a-\sqrt{a^2-1}),
$$
or, inversely,
\begin{equation}\label{recip}
a=J(\cos\pi \omega_0).
\end{equation}

Now we compute the corresponding quantity for $Q_2$. 
Do do this, we map $Q_2$ conformally onto a region $Q_3$
with double
symmetry, namely on the unit disk from which two disks of equal radii $r$
tangent from inside at $\pm1$ are removed.
This mapping is performed by the real fractional linear transformation 
\begin{equation}\label{psi}
\psi(z)=\frac{\sqrt{2}z-2}{\sqrt{2}z+2}
\end{equation}
which satisfies
$$\psi(0)=-1,\quad \psi(\infty)=1,\quad \psi(2)=-\psi(1),$$
so that
$$r=\frac12(1-\psi(2))=\sqrt{2}-1\approx 0.414214.$$
Now $\omega_0$ is the harmonic measure
of the circle $|z+1-r|=r$ at the center~$0$.
This is computed numerically, using the Schwarz Alternating Method~\cite{HC}. 
Usually this method is applied to a union of regions, but a
proper modification also works for the intersection of regions.
(To the best of our knowledge, a closed form formula for the modulus of $Q_2$
is not known and probably does not exist).

Let us denote by $L=\{ z:|z+1-r|=r\}$ the left small circle and
by $R=-L$ the right small circle. Then $Q_3$ is the region bounded by
the unit circle $\partial\U$ and the circles $L$ and $R$.
Thus $Q_3=G_L\cap G_R$, where $G_L$ is bounded by $L$ and $\partial\U$
and $G_R$ is bounded by $R$ and~$\partial\U$.

Now define two sequences $(u_k)$ and $(v_k)$  of harmonic functions:
\begin{itemize}
\setlength{\itemindent}{10pt}
\item[] $u_0$ is harmonic in $G_L$, equals $1$ on $L$ and $0$ on~$\partial\U$, 
\item[] $v_0$ is harmonic in $G_R$, equals $u_0$ on $R$, and $0$ on~$\partial\U$,
\item[] $u_1$ is harmonic in $G_L$, equals $v_0$ on $L$, and $0$ on~$\partial\U$,
\end{itemize}
and so on.
Thus, in general, 
\begin{itemize}
\setlength{\itemindent}{10pt}
\item[] $v_k$ is harmonic in $G_R$, equals $u_k$ on~$R$, and zero on~$\partial\U$,
\item[] $u_{k+1}$ is harmonic in $G_L$, equals $v_k$ on $L$ and $0$ on~$\partial\U$.
\end{itemize}
All these functions can be computed using the explicit Poisson formulas
which are available for $G_R$ and $G_L$ (see, for example,~\cite{LS}). 

Now a standard argument shows that the alternating series
$$\omega=u_0-v_0+u_1-v_1+u_2-v_2+\ldots$$
converges uniformly on compact subsets of $G$
and satisfies the boundary conditions
$\omega(z)=1$, $z\in L$, and $\omega(z)=0$
on the rest of the boundary.
Moreover, this series is alternating, so we have an automatic rigorous error
control. The speed of convergence is geometric.
The computation gives $\omega_0\approx0.483903$. 

Substituting this value in (\ref{recip}) gives the value 
\begin{equation}\label{anum}
a\approx 9.91706\; .
\end{equation}
Now  the value $\mu\approx 0.0252896$ follows from \eqref{q} and \eqref{mu}.

To obtain $6$ significant digits, $20$ iterations of the Schwarz method
were used. The computation was performed with Maple 14.
Matti Vuorinen, Harri Hakula and Antti Rasila
verified this 
computation with a different algorithm.
Our Maple script is available on www.math.purdue.edu/$\sim${eremenko}.

\section{Computation of the extremal function $h$}\label{last}

The contents of this section
is close to the papers~\cite{BP}
and~\cite{E}, studying conformal maps of circular polygons.

We recall that $h$ is given by the formula (\ref{defh}),
where $\chi$ is a fractional-linear transformation (\ref{chi}),
$J$ is the Joukowski function, and $\tau$ is defined in (\ref{tau}).
For the modular function $\Lambda$, explicit expressions are known
(see, for example, \cite{Akh,HC}) and the constant $\mu$ in (\ref{chi})
has been computed in the previous section.

It remains to compute $\phi$ in (\ref{tau}).
Instead we will compute $\theta:=\psi\circ\phi:\H\to Q^*$,
where 
$$Q^*=\{ z:|z|<1,\; \Ima z>0,\; |z-1+r|>r, \;|z+1-r|>r\},\quad
r=\sqrt{2}-1,$$
and $\psi$ is the fractional-linear transformation (\ref{psi}).
The boundary correspondence of $\theta$ is the following:
$$\theta:(\infty,-a,-1,1)\mapsto(1,-1,-1+2r,1-2r).$$
where $a$ has been defined in (\ref{defa}) and numerically
computed in (\ref{anum}).

According to the general theory of conformal mapping of polygons 
bounded by arcs of circles (see \cite[Section III.7.7]{HC}), our function
$\theta$ is a solution of the Schwarz differential equation
$$\{\theta,z\}:=\frac{\theta^{\prime\prime\prime}}{\theta^\prime}-
\frac{3}{2}\left(\frac{\theta^{\prime\prime}}{\theta^{\prime}}\right)^2=
\frac{3}{4}\frac{z^2+1}{(z^2-1)^2}+\frac{1}{2(z+a)^2}+\frac{c_1}{z-1}+
\frac{c_{-1}}{z+1}+\frac{c_a}{z+a}.$$
where
$c_1,c_{-1}$ and $c_a$ are the real accessory parameters which satisfy
two relations
$$c_a+c_1+c_{-1}=0\quad\mbox{\and}\quad
ac_a+c_1+c_{-1}=\frac34,$$
coming from the condition that the angle
corresponding to $\infty$ is zero.
Thus 
$$c_a=\frac{3}{4(a-1)}\quad\mbox{and}\quad c_1+c_{-1}=-\frac{3}{4(a-1)}.$$
One real parameter, say $c_1$, remains.
Any real solution $\theta$ of this equation which satisfies
$\theta(0)\in\R$ and $\theta'(0)>0$ will map the upper half-plane onto
a quadrilateral, with interior angles $(0,0,\pi/2,\pi/2)$ at the images
of $(\infty,-a,-1,1)$, and the interval $(-1,1)$ will
be mapped on the real line.
One has to choose the remaining accessory
parameter and normalization of $\theta$,
so that the vertices with zero angles are at $-1,1$, and the other two
vertices are symmetric with respect to $0$ on the interval $(-1,1)$.
Then our choice of $a$ and $r$ in the previous section guarantees
that the image of $\phi$ is $Q^*$, and the boundary correspondence is
correct.

To prove that the remaining accessory parameter with the stated
properties indeed exists and to obtain a numerical algorithm that finds
it, we perform one additional conformal mapping, to explore the symmetry
of the problem.

The Schwarz--Christoffel map
\begin{equation}\label{elliptic}
\varphi(z)=
C\int_{-1}^z\frac{d\zeta}{\sqrt{(\zeta+a)(1-\zeta^2)}}-\omega,
\end{equation}
where
$C$ is chosen from the condition that $\varphi(-a)=\pi i-\omega$,
that is
$$C^{-1}=
-\frac{1}{\pi i}\int_{-a}^{-1}\frac{d\zeta}{\sqrt{(\zeta+a)(1-\zeta^2)}},$$
and 
$$\omega=\frac{C}{2}\int_{-1}^1\frac{d\zeta}{\sqrt{(\zeta+a)(1-\zeta^2)}},$$
maps $\H$ onto the rectangle
$$R^*=\{ x+iy:-\omega<x<\omega,\; 0<y<\pi\}.$$
The function $\sigma=\theta\circ\varphi^{-1}$ maps the rectangle onto
our region $Q^*$. By symmetry, it also maps the right half $R$ of our
rectangle onto the right half $Q$ of $Q^*$.
It is this map 
$$\sigma:R\to Q$$
that we are going to compute; cf. Figure~\ref{z}. A similar problem was solved
in~\cite{E}.
\begin{figure}[htb]
\begin{center}
\includegraphics[height=6cm]{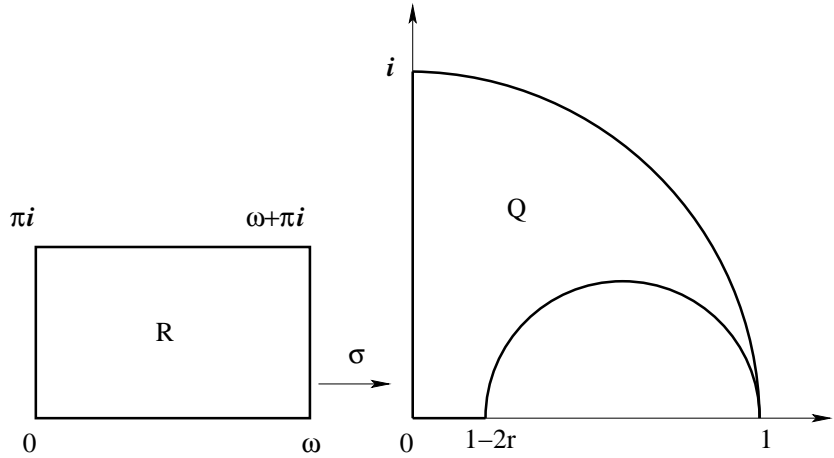}
\caption{The conformal map $\sigma:(0,\omega,\omega+\pi i,\pi i)\mapsto
(0,1-2r,1,i)$.}
\label{z}
\end{center}
\end{figure}

Let $\wp$ be the Weierstra{\ss} elliptic 
function with periods $2\omega$ and $2\pi i$.
We use the standard notation of the theory of elliptic functions as
in~\cite{Akh,HC}.
We put 
$$P(z)=\frac{1}{4}\left(\wp(z+\omega+i\pi)-e_2\right).$$
Then $P$ is real on both real and imaginary axis, in fact it maps our
rectangle $R$ onto the lower half-plane. The 
function $P$ is holomorphic in the
closure of~$R$, except one point $\omega+\pi i$, where it has a pole of
the second order.

Our function $\sigma$ will be the ratio of two linearly independent
solutions of the Lam\'e differential equation
\begin{equation}\label{Q1}
\frac{d^2w}{dz^2}+Pw=\lambda w,
\end{equation}
where $\lambda$ is an accessory parameter to be chosen. 

Now we describe the choice of $\lambda$.
Consider the differential equation obtained by the change of variable
$z=it$:
\begin{equation}\label{Q3}
\frac{d^2w}{dt^2}=(P(it)-\lambda)w.
\end{equation}
Let $\lambda_0$ be the smallest eigenvalue of (\ref{Q3}) with
the boundary condition $w'(0)=w'(\pi)=0$.
As $P(z)>0$ for $z\in(0,i\pi)$, we conclude from the Sturm comparison 
theorem \cite[Chapter~X]{Ince} that $\lambda_0\geq 0$.
Let $\lambda_2$ be the smallest eigenvalue of (\ref{Q3}) with
the boundary conditions $w'(0)=0$, $w(\pi)=0$. By Sturm's theory
we have $\lambda_2>\lambda_0$.

Let $c(\lambda,z)$ and $s(\lambda,z)$ be two linearly independent
solutions of (\ref{Q1}) normalized by the condition
$$\left(\begin{array}{cc}c(0)&s(0)\\c'(0)&s'(0)\end{array}\right)=\left(
\begin{array}{cc}1&0\\0&1\end{array}\right).$$
Then
\begin{equation}\label{wronsk}
cs'-c's=1,
\end{equation}
where the primes stand for differentiation with respect to~$z$.
By our choice of $\lambda_0$ we have 
$$c'(\lambda_0,i\pi)s(\lambda_0,i\pi)=0.$$
Together with (\ref{wronsk}) this gives
\begin{equation}\label{one}
c(\lambda_0,i\pi)s'(\lambda_0,i\pi)=1.
\end{equation}
By our choice of $\lambda_2$ we have
\begin{equation}\label{two}
c(\lambda_2,i\pi)s'(\lambda_2(i\pi)=0.
\end{equation}
The equations (\ref{one}) and (\ref{two}) imply that there exists
$\lambda_1\in(\lambda_0,\lambda_2)$ such that
\begin{equation}\label{half}
c(\lambda_1,i\pi)s'(\lambda_1,i\pi)=1/2.
\end{equation}
This $\lambda_1$ is our choice of the accessory parameter in (\ref{Q1}).
We will later see that it is unique.
For the numerical computation, we solve equation (\ref{half}) on
the interval $(\lambda_0,\lambda_2)$ by a simple dissection method.

Now we prove that $\sigma(z)=K s(z)/c(z),$
with a real constant~$K$. From now on $\lambda_1$ is fixed,
and we don't write it in the formulas.
Combining 
(\ref{wronsk}) and (\ref{half}) we obtain
\begin{equation}\label{QM}
c(i\pi)s'(i\pi)+c'(i\pi)s(i\pi)=0.
\end{equation}
The function $f:=s/c$ is locally univalent
in $\overline{R}\backslash\{\omega+i\pi\}$
as a solution of the Schwarz equation
$$\frac{f^{\prime\prime\prime}}{f^\prime}
-\frac{3}{2}\left(\frac{f^{\prime\prime}}{f^\prime}\right)^2=2(P-\lambda_1).$$
The functions $c$ and $s$ are real on $[0,\omega]$, because they satisfy
a real differential equation and real initial conditions.
For the same reason, $c$ is real and $s$ is purely imaginary on $[0,i\pi]$.

We claim that $c$ has no zeros on the sides $[0,\omega]$ and $[0,i\pi]$.
On $[0,i\pi]$ this follows from our choice $\lambda_1<\lambda_2$. Indeed,
Sturm's theory implies that $c$ cannot have zeros on $[0,i\pi]$
for $\lambda<\lambda_2$. On $[0,\omega]$ we notice that $P<0$,
and $\lambda_1>\lambda_0\geq 0$, so the solution $c$ with $c(0)=1$
cannot have zeros on $[0,\omega]$. This proves the claim.

We have $f(0)=0$ and $f$ is increasing near $0$ and locally univalent
on $[0,\omega]$, so it maps $[0,\omega]$ on some interval $[0,p]$ bijectively.
The same applies to $[0,i\pi]$ which is mapped on some interval $[0,iq]$
bijectively.
The image of the vertical side $[\omega,\omega+i\pi]$ of
the rectangle $R$ must be an arc of a circle $C_1$ perpendicular
to the real line. To see this, we consider a pair of linearly independent
solutions $u,v$ of (\ref{Q1}) normalized by
$$\left(\begin{array}{cc}u(\omega)&v(\omega)\\
u'(\omega)&v'(\omega)\end{array}\right)=
\left(\begin{array}{cc}1&0\\0&1\end{array}\right).$$
Then $u$ is real and $v$ is purely imaginary on $[\omega,\omega+i\pi]$,
for the same reason that $c,s$ are real and imaginary on the imaginary line,
and we have
\begin{eqnarray*}c&=&c(\omega)u+c'(\omega)v,\\
                 s&=&s(\omega)u+s'(\omega)v.
\end{eqnarray*}
It follows that $f=s/c$ maps the side $[\omega,\omega+i\pi]$ injectively
into
the circle
$$\left\{\frac{s(\omega)+s'(\omega)it}{c(\omega)+c'(\omega)it}:t\in\R\right\}.$$
As the circle is perpendicular to the real line, its center lies
on the real line. It is easy to see that the location of the center is 
$$\frac{1}{2}\left(\frac{s(\omega)}{c(\omega)}+
\frac{s'(\omega)}{c'(\omega)}\right).$$
Similarly, $f=s/c$ maps the horizontal side $[i\pi,\omega+i\pi]$ injectively
into a circle $C_2$ perpendicular to the imaginary line
whose center is located at
$$\frac{1}{2}\left(\frac{s(i\pi)}{c(i\pi)}+
\frac{s'(i\pi)}{c'(i\pi)}\right).$$
Now equation (\ref{QM}) implies that the center of this circle is at the
origin. The two circles must have a common point at $f(\omega+i\pi)$
and they must be tangent at this point, because of
the form of the Schwarz equation
(\ref{Q1}) near this point.
Thus $f$ maps $R$ onto a quadrilateral bounded by a vertical side $[0,iq]$,
a horizontal side $[0,p]$ and two circles, perpendicular to the axes which
are tangent at one point. Clearly, this tangent point
must be on the real line. As the modulus of the quadrilateral $R$
is the same as the modulus of the quadrilateral~$Q$, by our choice 
of the constants $a,\omega$ and~$r$, we conclude that
$f(R)$ is similar to~$Q$, and it remains to multiply $f$ by a constant factor
to obtain the function~$\sigma$.

Thus we have represented our extremal function $h$ as a composition
of  the fractional linear transformations $\chi$ and $\psi$ given
in (\ref{chi}) and (\ref{psi}),
the Joukowski function~$J$, an
elliptic integral $\varphi$ in (\ref{elliptic}), a solution of
the Schwarz equation which is the ratio of two solutions
of the Lam\'e equation equation (\ref{Q1}), and the modular function~$\Lambda$.

\section{Proof of Theorem \ref{three}}
\label{proof67}
We begin with the following lemma.
\begin{lemma}\label{lemma6}
Let $z_1,\ldots,z_k\in \U$, $m_1,\ldots,m_k\in \N$, 
$f:\U\to\C$ holomorphic and $\varepsilon>0$.
Then there exists a polynomial $P$ satisfying
$f^{(m)}(z_j)=P^{(m)}(z_j)$ for $1\leq j\leq k$ 
and $0\leq m\leq m_j$ such that 
$|P(z)-f(z)|<\varepsilon$ for $|z|<1-\varepsilon$.

If, in addition, $\{z_1,\ldots,z_k\}=f^{-1}(S)$ for 
some $S\subset \C$ and $m_j$ is the multiplicity of 
$f$ at $z_j$, then
$P$ may be chosen such that 
$P(z)\notin S$ if 
$|z|<1-\varepsilon$ and $z\notin \{z_1,\ldots,z_k\}$.
\end{lemma}
\begin{proof}
There exists a polynomial $Q$
satisfying $f^{(m)}(z_j)=Q^{(m)}(z_j)$ for $1\leq j\leq k$
and $0\leq m\leq m_j$.
Let $R(z)=\prod_{j=1}^k (z-z_j)^{m_j}$. Then 
$(f-Q)/R$ is holomorphic in $\U$ and 
thus the sequence $(T_k)$ of Taylor polynomials 
converges locally uniformly in $\U$ to $(f-Q)/R$.
With $P_k=T_kR+Q$ we find that $(P_k)$ converges 
locally uniformly to~$f$.
Moreover, 
$P_k^{(m)}(z_j)=Q^{(m)}(z_j)=f^{(m)}(z_j)$ for $1\leq j\leq k$ 
and $0\leq m\leq m_j$. 

Taking $P=P_k$ for sufficiently large $k$ we obtain the 
first conclusion. 
The second conclusion follows  from Hurwitz's theorem.
\end{proof}

\begin{proof}[Proof of Theorem \ref{three}]
The quotient on the left-hand side of (\ref{ineq}) remains unchanged if 
$a$ and $b$ are replaced by $\phi(a)$ and $\phi(b)$ for 
some automorphism $\phi$ of~$\U$.
Thus we may assume without loss of generality that
$-b=a>0$.
The necessity of the condition
(\ref{ineq}) now  follows from Theorem~\ref{wk}.
It also follows from Theorem~\ref{wk} that
equality cannot hold in (\ref{ineq})
for a rational function.

Conversely, our function $h$ shows that (\ref{ineq})
is sufficient for the existence of a holomorphic function 
$f:\U\to\C$ satisfying $f^{-1}(\{0,1\})=\{a,b\}$, and Lemma~\ref{lemma6}
shows that if we have strict inequality in (\ref{ineq}),
then there even exists a polynomial with this property.
\end{proof}

\section{The Belgian Chocolate Problem}\label{bcp}
We consider a question posed by Blondel~\cite[p.~149f]{Blondel}
which is  known as the ``Belgian Chocolate Problem''.
We follow~\cite{Burke} in our formulation of this problem:

{\em  Let 
$a(z)=z^2-2\delta z+1$ and $b(z)=z^2-1$. 
For which $\delta\in (0,1)$
do there exist stable (real) polynomials $p$ and $q$ 
with $\deg(p)\geq\deg(q)$ such that $ap+bq$ is stable?}

Here a polynomial is called {\em   stable} if all its roots
are in the left half-plane. 
It is known that there exists $\delta^*$ such that $p$ and
$q$ as required exist for $0<\delta<\delta^*$ and do not exist 
for $\delta^*\leq \delta<1$.

If $p$ and $q$ are as above, then  the rational function
$R=bq/(ap+bq)$ satisfies 
$R(1)=0$ and $R(\delta\pm i\sqrt{1-\delta^2})=1$, and 
all other $0$- and $1$-points and all poles of $R$ are 
in the  left half-plane.
Passing from the left half-plane to the unit disk by a fractional
linear transformation and using Lemma~\ref{lemma6} we see 
that the above problem is equivalent to the following one: 

{\em For which $t>0$ does there exist a real  holomorphic
function $f:\U\to\C$ having a simple zero at~$0$, simple $1$-points
at $\pm it$, and no other $0$- or $1$-points in~$\U$?}

We find that there exists $t^*$ such that a function $f$
with these properties exists 
for $t^*\leq t<1$ and does not exist for $0<t<t^*$.
The numbers $t^*$ and $\delta^*$ are related by
$$t^*=\sqrt{\frac{1-\delta^*}{1+\delta^*}}
\quad\text{and}\quad 
\delta^*=\frac{1-{t^*}^2}{1+{t^*}^2}.$$

Clearly, we have $t^*\geq A_2$, and Batra's~\cite{Batra}
estimate $A_2\geq 0.0012$ also seems to be the best lower 
bound for $t^*$ obtained previously.
Theorem~\ref{AAA} yields $t^*\geq 0.00587$.
We can further improve this  bound as follows.
\begin{theorem}\label{chocolate}
\quad $t^*> 0.01450779$.
\end{theorem}
In terms of
$\delta^*$ this takes the form $\delta^*< 0.999579$.
The best previously known upper bound was
$\delta^*< 0.99999712$, obtained from Batra's~\cite{Batra}
estimate $t^*\geq A_2\geq 0.0012$.
(The frequently cited~\cite{Burke,ChangSahinidis,Patel} 
upper bound $\delta^*< 0.9999800002$
seems to come from a computational error using 
the lower bound $A_2\geq 10^{-5}$ given in~\cite{1}.)

The best known lower bound for $\delta^*$ is 
$\delta^*> 0.973974$; cf.~\cite{ChangSahinidis}.
In terms of $t^*$ this takes the form $t^*<0.114825$.

\begin{proof}[Proof of Theorem \ref{chocolate}]
Let $f:\U\to\C$ be a holomorphic function satisfying $f(0)=0$ and
$f(\pm it)=1$, having no further $0$- or $1$-points in~$\U$.
Theorem~${1.1^\prime}$ and the Schwarz Lemma yield
that 
$\ell(\U\setminus \{0,\pm it\})\geq 2\log(3+2\sqrt{2})$.
As $z\mapsto -z^2$ is a covering map from
$\U\setminus \{0,\pm it\}$ onto $\U\setminus \{0,t^2\}$,
we have
$2\ell(\U\setminus \{0,\pm it\})=\ell(\U\setminus  \{0,t^2\})$.
Thus
\begin{equation}\label{choco2}
\ell(\U\setminus  \{0,t^2\})\geq \log(3+2\sqrt{2}).
\end{equation}
We have 
$\ell(\U\setminus  \{0,t^2\})=\ell(\U\setminus  \{-s,s\})$
with $a$ and $t$ related by $t^2=2s/(1+s^2)$.
Thus
$$
\ell(\U\setminus  \{-s,s\})
\geq \log(3+2\sqrt{2}).
$$
The equation
$\ell(\U\setminus  \{-s_0,s_0\})= \log(3+2\sqrt{2})$ 
is of the same type as~\eqref{ellmu} and
can be solved
numerically with the method used 
in section~\ref{A} to compute~$\mu$,
 or the one described in~\cite[Section~5]{BP}.
We obtain $s_0\approx 0.0001054752$ and this implies that
$t^*\geq \sqrt{2s_0/(1+s_0^2)} > 0.01450779$.
\end{proof}
\begin{remark}
A slightly weaker lower bound for $t^*$ can be obtained
without computer assistance.
Hempel and Smith~\cite[inequality (9)]{HS}
showed that 
\begin{equation}\label{choco3}
\ell(\U\setminus  \{0,r\})\leq
\frac{2\pi^2}{\log\left(16\sqrt{1-r}/r\right) -
\pi^2\left(4\log\left(16 \sqrt{1-r}/r\right) \right)^{-1}}
\end{equation}
for $0<r<1$. Using \eqref{choco2} and \eqref{choco3} with
$r=t^2$ we can then show that $t^*>0.0132889$.
\end{remark}

\begin{ack}
We thank Peter Buser who brought~\cite{Baribaud} to our attention
and explained it,
Vladimir Dubinin and Alexander Solynin for drawing our attention
to~\cite{Dubinin},
 Andrei Gabrielov for useful discussions, and Matti Vuorinen,
Harri Hakula and Antti Rasila 
for verification of the numerical computation
of section 7 with their
own method.
\end{ack}

\end{document}